\documentclass{article}
\usepackage{amsmath}
\usepackage{amssymb}
\usepackage{amsfonts}

\newtheorem{theorem}{Theorem}

\newtheorem{corollary}[theorem]{Corollary}

\newenvironment{proof}[1][Proof]{\noindent\textbf{#1.} }{\ \rule{0.5em}{0.5em}}
\input{tcilatex}

\begin{document}

\title{Mutually inverse series relating Ferrers and associated Legendre
functions\ and generating functions pertaining to them}
\author{Pinchas Malits \\
(PERI, Physics and Engineering Research Institute\\
at Ruppin Academic Center, Emek Hefer 40250,\\
Israel)}
\maketitle


\section{\textbf{Introduction}}

Ferrers functions of the first kind $P_{\nu }^{-\mu }\left( x\right) $ are
solutions of the second order differential equation%
\begin{equation}
\frac{d}{dx}\left( \left( 1-x^{2}\right) \frac{dy}{dx}\right) +\left( \nu
\left( \nu +1\right) -\frac{\mu ^{2}}{\left( 1-x^{2}\right) }\right) y=0%
\text{, }\nu ,\mu \in \mathbb{C}\text{,}  \label{Leg-eq}
\end{equation}%
on the interval $-1<x<1$, and associated Legendre functions of the first
kind $P_{\nu }^{-\mu }\left( x\right) $ are solutions of the same equation
on the interval $x>1$ \cite{Erdelyi1,NIS}. They typically arise in numerous
problems of mathematical physics as a result of separation of variables in
various coordinate systems \cite{Miller1} while their polynomial cases also
play an important role in approximation theory and numerical analysis \cite%
{Atki,G}. In many cases, it is necessary to express or re-expand one of the
above-mentioned functions in terms of the other. Results of this article can
be helpful in solving such problems.

Functions $P_{\nu }^{-\mu }\left( x\right) $ are expressed by means of the
Gauss hypergeometric function $F\left( \alpha ,\beta ;\gamma ;t\right) $:
for $x>1$ in \cite[p.122, 3.2(3)]{Erdelyi1} and for $\left\vert x\right\vert
<1$ in \cite[p.143, 3.4(6)]{Erdelyi1}. For the sake of convenience, we write
both expressions in the form of the unified formula

\begin{equation}
P_{\nu }^{-\mu }\left( x\right) =P_{-1-\nu }^{-\mu }\left( x\right)
=\left\vert \frac{1-x}{1+x}\right\vert ^{\frac{\mu }{2}}\frac{F\left( -\nu
,\nu +1;1+\mu ;\frac{1-x}{2}\right) }{\Gamma \left( 1+\mu \right) }\text{, }%
x>-1\text{,}  \label{1}
\end{equation}%
or on using the Pfaff transformation for Gauss functions \cite[p.64, 2.1(22)]%
{Erdelyi1} 
\begin{equation}
P_{\nu }^{-\mu }\left( x\right) =P_{-1-\nu }^{-\mu }\left( x\right) =\frac{%
2^{-\nu }\left\vert 1-x\right\vert ^{\frac{\mu }{2}}}{\Gamma \left( 1+\mu
\right) \left( 1+x\right) ^{\frac{\mu }{2}-\nu }}F\left( -\nu ,\mu -\nu
;1+\mu ;\frac{x-1}{x+1}\right) \text{.}  \label{Hyper1}
\end{equation}%
It is readily seen that Ferrers and associated Legendre functions of the
first kind are entire functions in both parameters $\mu $ and $\nu $.

Note, that for $x>1$ the second linearly independent solution of (\ref%
{Leg-eq}), an associated Legendre functions of the second kind $Q_{\nu
}^{-\mu }\left( x\right) $, can be also expressed in the terms of associated
Legendre functions of the first kind of the argument belonging to $\left(
1,\infty \right) $ , namely \cite[p.141, 3.3 (13)]{Erdelyi1} 
\begin{equation}
Q_{\nu }^{-\mu }\left( x\right) =\sqrt{\frac{\pi }{2}}e^{-i\pi \mu }\frac{%
\Gamma \left( \nu -\mu +1\right) }{\left( x^{2}-1\right) ^{\frac{1}{4}}}%
P_{\mu -\frac{1}{2}}^{-\nu -\frac{1}{2}}\left( \frac{x}{\sqrt{x^{2}-1}}%
\right) \text{, }x>1\text{.}  \label{Q}
\end{equation}

For $\nu =k+\mu $ or $\nu =-k$, $k\in \mathbb{N}_{0}$, the Gauss
hypergeometric series in (\ref{Hyper1}) turns into Jacobi polynomials. In
particular, 
\begin{equation}
P_{k+\mu }^{-\mu }\left( x\right) =\frac{2^{\mu }k!\Gamma \left( \mu
+1/2\right) }{\sqrt{\pi }\Gamma \left( 2\mu +k+1\right) }\left\vert
1-x^{2}\right\vert ^{\frac{\mu }{2}}C_{k}^{\mu +1/2}\left( x\right) \text{, }%
x>-1\text{,}  \label{P-C}
\end{equation}%
where%
\begin{equation}
C_{k}^{\tau }\left( x\right) =\left( -1\right) ^{k}C_{k}^{\tau }\left(
-x\right) =\frac{1}{\Gamma \left( \tau \right) }\sum_{j=0}^{\left[ \frac{k}{2%
}\right] }\frac{\left( -1\right) ^{j}\Gamma \left( k-j+\tau \right) }{%
j!\left( k-2j\right) !}\left( 2x\right) ^{k-2j}\text{.}  \label{Ck}
\end{equation}%
are Gegenbauer polynomials (see also \cite{Erdelyi2,Szego}).

In the important partial case $\nu =k$, $\mu =\pm m$, $k,m\in \mathbb{N}_{0} 
$, we have the associated Legendre polynomials $P_{k}^{\pm m}\left( x\right) 
$ given as $k\geq m$ by

\begin{equation}
P_{k}^{m}\left( x\right) =\left( -1\right) ^{m}\frac{\left( k+m\right) !}{%
\left( k-m\right) !}P_{k}^{-m}\left( x\right) =\frac{\left( -1\right)
^{m}\left( 2m\right) !}{2^{m}m!}\left\vert 1-x^{2}\right\vert ^{\frac{m}{2}%
}C_{k-m}^{m+1/2}\left( x\right) \text{.}  \label{Pn}
\end{equation}%
As $k<m$, $P_{k}^{m}\left( x\right) =0$.

Certain special choices of parameters and transformation formulas for
hypergeometric functions \cite{Erdelyi1} allow us to obtain several
relations between Ferrers and associated Legendre functions. In particular,
by virtue of (\ref{Hyper1}) we obtain on setting $\mu =2\nu +1$, $x=\left(
1+t\right) /\left( 1-t\right) $\ into (\ref{1}) 
\begin{equation}
P_{\nu }^{-2\nu -1}\left( \frac{1+t}{1-t}\right) =\sqrt{1-t}P_{\nu }^{-2\nu
-1}\left( 1-2t\right) \text{, }0<t<1\text{,}
\end{equation}%
while setting $\ \nu =1/2$, $x=\left( 1+t\right) /\left( 2\sqrt{t}\right) $
into \cite[ 3.2(28)]{Erdelyi1} \ 
\begin{equation*}
P_{\nu }^{-\mu }\left( x\right) =\frac{2^{-\mu }\left( x^{2}-1\right) ^{%
\frac{\mu }{2}}}{\Gamma \left( 1+\mu \right) \left( x-\sqrt{x^{2}-1}\right)
^{\mu -\nu }}F\left( \mu -\nu ,\mu +\frac{1}{2};1+2\mu ;\frac{2\sqrt{x^{2}-1}%
}{\sqrt{x^{2}-1}-x}\right)
\end{equation*}%
and $\nu =\mu -1/2$, $x=1/\sqrt{1-t}$ into \cite[ 3.2(20)]{Erdelyi1}

\begin{equation*}
P_{\nu }^{-\mu }\left( x\right) =\frac{2^{-\mu }\left( x^{2}-1\right) ^{%
\frac{\mu }{2}}}{\Gamma \left( 1+\mu \right) }F\left( \frac{\mu +\nu +1}{2},%
\frac{\mu -\nu }{2};1+\mu ;1-x^{2}\right)
\end{equation*}%
yield%
\begin{equation}
P_{\frac{1}{2}}^{-\mu }\left( \frac{1+t}{2\sqrt{t}}\right) =2^{-\mu }P_{\mu -%
\frac{1}{2}}^{-2\mu }\left( 2t-1\right) ,0<t<1\text{,}
\end{equation}%
\begin{equation}
P_{-\frac{1}{2}}^{-\mu }\left( \frac{1}{\sqrt{1-t}}\right) =\sqrt[4]{1-t}P_{-%
\frac{1}{4}}^{-\mu }\left( 1-2t\right) ,0<t<1\text{.}
\end{equation}%
There are additional similar relations (some of them readers can find in 
\cite{Maier1,Maier2}) but such simple relations can not be expected for
arbitrary orders and indices. In sections 4 and section 5, we derive a
number of pairs of mutually inverse series relating Ferrers and associated
Legendre functions of arbitrary $\mu ,\nu \in \mathbb{C}$ which under
special choices of $\mu $ and $\nu $ turn into\ mutually inverse sums
relating Gegenbauer or associated Legendre polynomials of different
arguments or into certain connection formulas for Gegenbauer and associated
Legendre polynomials. Our derivations are based on integral representations
of $P_{\nu }^{-\mu }\left( x\right) $ that are given in section 2 and on
generating functions for certain polynomials that are discussed in section
3. In section 3, we also obtain high-order asymptotics for above-mentioned
polynomials. In particular, a uniform asymptotics for $C_{n}^{\lambda
-n}\left( x\right) $, $\lambda \in \mathbb{C}$, $n\rightarrow \infty $, is
established.

\section{Integral representations and asymptotics}

Our study will be based on integrals%
\begin{equation}
I_{1}^{-}\left( \alpha ,\mu ,\nu \right) =\int_{\alpha }^{\infty }\frac{%
e^{-t\mu }dt}{\left( \cosh t-\cosh \alpha \right) ^{-\nu }}\text{, }
\label{rep1}
\end{equation}%
\begin{equation}
I_{1}^{+}\left( \alpha ,\mu ,\nu \right) =\int_{\alpha }^{\infty }\frac{%
e^{-t\mu }dt}{\left( \sinh t-\sinh \alpha \right) ^{-\nu }}\text{, }
\label{rep2}
\end{equation}%
\begin{equation}
I_{2}^{-}\left( \alpha ,\mu ,\nu \right) =\int_{\alpha }^{\infty }\frac{%
e^{-2t\mu }\sinh ^{\nu }t}{\sinh ^{-\nu }\left( t-\alpha \right) }dt\text{, }
\label{rep3}
\end{equation}%
\begin{equation}
I_{2}^{+}\left( \alpha ,\mu ,\nu \right) =\int_{\alpha }^{\infty }\frac{%
e^{-2t\mu }\cosh ^{\nu }t}{\sinh ^{-\nu }\left( t-\alpha \right) }dt\text{, }
\label{rep4}
\end{equation}%
\begin{equation}
I_{3}^{+}\left( \alpha ,\sigma ,\nu \right) =\int_{\alpha }^{\infty }\frac{%
\sinh ^{\sigma }\left( t-\alpha \right) }{\cosh ^{\sigma +2\nu +2}t}dt\text{%
, }  \label{rep5}
\end{equation}%
\begin{equation}
I_{3}^{-}\left( \alpha ,\sigma ,\nu \right) =\int_{\alpha }^{\infty }\frac{%
\sinh ^{\sigma }\left( t-\alpha \right) }{\sinh ^{\sigma +2\nu +2}t}dt\text{,%
}  \label{rep6}
\end{equation}%
which for real $\alpha $ are convergent if\ $\func{Re}\mu >\func{Re}\nu $, $%
\func{Re}\nu >-1$, and $\func{Re}\sigma >-1$.

Show that the integrals $I_{1}^{+}\left( \alpha ,\mu ,\nu \right) $ for $%
\alpha \in \mathbb{R}$ give integral representations of Ferrers functions
and the integrals $I_{1}^{-}\left( \alpha ,\mu ,\nu \right) $ give for $%
\alpha >0$ integral representations of associated Legendre functions.

The first and second integrals can be readily evaluated by making the change 
$e^{-t}=e^{-\alpha }s$ and employing Euler's integral representation of the
Gauss hypergeometric function\ \cite{Erdelyi1}:%
\begin{eqnarray*}
I_{1}^{\pm }\left( \alpha ,\mu ,\nu \right) &=&\frac{e^{-\left( \mu -\nu
\right) \alpha }}{2^{\nu }}\int_{0}^{1}\frac{s^{\mu -\nu -1}\left(
1-s\right) ^{\nu }}{\left( 1\pm e^{-2\alpha }s\right) ^{-\nu }}ds \\
&=&\frac{\Gamma \left( 1+\nu \right) \Gamma \left( \mu -\nu \right) }{2^{\nu
}\Gamma \left( \mu +1\right) e^{\alpha \left( \mu -\nu \right) }}F\left(
-\nu ,\mu -\nu ;1+\mu ;\frac{\tanh ^{\pm 1}\alpha -1}{\tanh ^{\pm 1}\alpha +1%
}\right) \text{.}
\end{eqnarray*}%
Then, (\ref{Hyper1}) leads to%
\begin{equation}
I_{1}^{-}\left( \alpha ,\mu ,\nu \right) =\Gamma \left( 1+\nu \right) \Gamma
\left( \mu -\nu \right) \sinh ^{\nu }\alpha P_{\nu }^{-\mu }\left( \coth
\alpha \right) \text{,}  \label{rep11}
\end{equation}%
\begin{equation}
I_{1}^{+}\left( \alpha ,\mu ,\nu \right) =\Gamma \left( 1+\nu \right) \Gamma
\left( \mu -\nu \right) \cosh ^{\nu }\alpha P_{\nu }^{-\mu }\left( \tanh
\alpha \right) .  \label{rep33}
\end{equation}%
The integrals $I_{2}^{\pm }\left( \alpha ,\mu ,\nu \right) $ and $I_{3}^{\pm
}\left( \alpha ,\mu ,\nu \right) $ can be \ evaluated by means of the change 
$e^{-2t}=e^{-2\alpha }s$.%
\begin{eqnarray*}
I_{2}^{\pm }\left( \alpha ,\mu ,\nu \right) &=&\frac{e^{-\left( 2\mu -\nu
\right) \alpha }}{2^{2\nu +1}}\int_{0}^{1}\frac{s^{\mu -\nu -1}\left(
1-s\right) ^{\nu }}{\left( 1\pm e^{-2\alpha }s\right) ^{-\nu }}ds \\
&=&\frac{\Gamma \left( 1+\nu \right) \Gamma \left( \mu -\nu \right) }{%
2^{2\nu +1}\Gamma \left( \mu +1\right) e^{\alpha \left( 2\mu -\nu \right) }}%
F\left( -\nu ,\mu -\nu ;1+\mu ;\frac{\tanh ^{\pm 1}\alpha -1}{\tanh ^{\pm
1}\alpha +1}\right)
\end{eqnarray*}%
and%
\begin{eqnarray*}
I_{3}^{\pm }\left( \alpha ,\mu ,\nu \right) &=&2^{2\nu +1}e^{-\left( 2\nu
+1\right) \alpha }\int_{0}^{1}\frac{s^{\nu }\left( 1-s\right) ^{\sigma }}{%
\left( 1\pm e^{-2\alpha }s\right) ^{\sigma +2\nu +2}}ds \\
&=&\frac{2^{2\nu +1}\Gamma \left( 1+\nu \right) \Gamma \left( 1+\sigma
\right) }{\Gamma \left( \sigma +2\nu +2\right) e^{\alpha \left( 2\nu
+1\right) }}F\left( c-a,-a;c+1;\frac{\tanh ^{\pm 1}\alpha -1}{\tanh ^{\pm
1}\alpha +1}\right)
\end{eqnarray*}%
with $c=\sigma +2\nu +1$, $a=-\nu -1$.

Then, on invoking (\ref{Hyper1})%
\begin{equation}
I_{2}^{+}\left( \alpha ,\mu ,\nu \right) =2^{-1-\nu }\Gamma \left( 1+\nu
\right) \Gamma \left( \mu -\nu \right) e^{-\alpha \mu }\cosh ^{\nu }\alpha
P_{\nu }^{-\mu }\left( \tanh \alpha \right) \text{,}  \label{rep222}
\end{equation}%
\begin{equation}
I_{2}^{-}\left( \alpha ,\mu ,\nu \right) =2^{-1-\nu }\Gamma \left( 1+\nu
\right) \Gamma \left( \mu -\nu \right) e^{-\alpha \mu }\sinh ^{\nu }\alpha
P_{\nu }^{-\mu }\left( \coth \alpha \right) \text{,}  \label{rep444}
\end{equation}%
\begin{equation}
I_{3}^{+}\left( \alpha ,\sigma ,\nu \right) =\frac{2^{\nu }\Gamma \left(
\sigma +1\right) \Gamma \left( 1+\nu \right) }{\cosh ^{\nu +1}\alpha }P_{\nu
}^{-\sigma -\nu -1}\left( \tanh \alpha \right) \text{,}  \label{rep55}
\end{equation}%
\begin{equation}
I_{3}^{-}\left( \alpha ,\sigma ,\nu \right) =\frac{2^{\nu }\Gamma \left(
1+\sigma \right) \Gamma \left( 1+\nu \right) }{\sinh ^{\nu +1}\alpha }P_{\nu
}^{-\sigma -\nu -1}\left( \coth \alpha \right) \text{.}  \label{rep66}
\end{equation}

The integral (\ref{rep11}) by virtue of (\ref{Q}) is the well-known integral
representation of the associated Legendre functions of the second kind \ 
\cite[p.155, 3.7(4)]{Hobson,Erdelyi1} that usually is derived in a more
tedious way. The integral representations of Ferrrers functions (\ref{rep33}%
), (\ref{rep222}) and (\ref{rep55}) are given by the author in \cite{Malits}%
. The integral representations of associated Legendre functions (\ref{rep444}%
) and (\ref{rep66}) are not indicated in sources known to the author.

Asymptotic expansions of the integrals $I_{1}^{\pm }\left( \alpha ,\mu ,\nu
\right) $ as $\func{Re}\mu \rightarrow \infty $ and $\nu =\nu \left( \mu
\right) =O\left( 1\right) $ can be readily obtained by the change $t=y+a$ \
and making use of Watson's lemma for Laplace integrals \cite{Miller}. In
particular, we obtain the leading terms 
\begin{eqnarray}
\Gamma \left( \mu -\nu \right) P_{\nu }^{-\mu }\left( \coth \alpha \right)
&=&\frac{e^{-\alpha \mu }}{\mu ^{\nu +1}}\left( 1+O\left( \frac{1}{%
\left\vert \mu \right\vert }\right) \right) \text{, }\alpha >0\text{,}
\label{as11} \\
\Gamma \left( \mu -\nu \right) P_{\nu }^{-\mu }\left( \tanh \alpha \right)
&=&\frac{e^{-\alpha \mu }}{\mu ^{\nu +1}}\left( 1+O\left( \frac{1}{%
\left\vert \mu \right\vert }\right) \right) \text{, }\alpha >0\text{.}
\label{as12}
\end{eqnarray}%
It follows from (\ref{1}) and asymptotics of gamma functions $\Gamma \left(
\eta +\gamma \right) /\Gamma \left( \eta +\delta \right) =\eta ^{\gamma
-\delta }\left( 1+O\left( 1/\left\vert \eta \right\vert \right) \right) $
that (\ref{as11}) and (\ref{as12}) are valid for any fixed number $\nu $.

Asymptotic expansions of $P_{\nu }^{-\sigma -\nu }\left( \coth \alpha
\right) $ and $P_{\nu }^{-\sigma -\nu }\left( \tanh \alpha \right) $ as $%
\func{Re}\nu \rightarrow +\infty $, $\func{Im}\nu =O\left( 1\right) $, and $%
\sigma =\sigma \left( \nu \right) =O\left( 1\right) $ follow from the
representation (see \cite[pp.128-129, 3.2(24)]{Erdelyi1} and \cite{Cohl1}) 
\begin{equation*}
P_{\nu }^{-\mu }\left( x\right) =\frac{2^{-\mu }\left\vert
1-x^{2}\right\vert ^{\frac{\mu }{2}}x^{\nu -\mu }}{\Gamma \left( 1+\mu
\right) }F\left( \frac{\mu -\nu }{2},\frac{\mu -\nu +1}{2};1+\mu ;1-\frac{1}{%
x^{2}}\right) \text{, }x>0\text{,}
\end{equation*}%
which leads to

\begin{eqnarray}
P_{\nu }^{-\sigma -\nu }\left( \coth \alpha \right) &=&\frac{\sinh ^{-\nu
}\alpha \cosh ^{-\sigma }\alpha }{2^{\sigma +\nu }\Gamma \left( \sigma +\nu
+1\right) }\left( 1+O\left( \frac{1}{\nu }\right) \right) \text{, }\alpha >0%
\text{,}  \label{as21} \\
P_{\nu }^{-\sigma -\nu }\left( \tanh \alpha \right) &=&\frac{\cosh ^{-\nu
}\alpha \sinh ^{-\sigma }\alpha }{2^{\sigma +\nu }\Gamma \left( \sigma +\nu
+1\right) }\left( 1+O\left( \frac{1}{\nu }\right) \right) \text{, }\alpha >0%
\text{.}  \label{as22}
\end{eqnarray}

\section{Some generating functions and large order asymptotics for related
functions}

The key role in derivations of this article are playing two families of
generating functions for certain functions expressed in terms of
hypergeometric or generalized hypergeometric polynomials.

The first family consists of partial cases of the generating function for
the Lauricella hypergeometric polynomials (see \cite[p.189, (6.2.1)]{Exton}
and \cite[p.383, 7.6(27)]{Sriv})%
\begin{eqnarray}
&&\sum_{n=0}^{\infty }\frac{\left( 1+\alpha +\beta n\right) _{n}}{n!}%
F_{D}^{\left( N\right) }\left( -n,\tau _{1},...,\tau _{N};1+\alpha +\beta
n;x_{1},...,x_{N}\right) t^{n}  \notag \\
&=&\frac{\left( 1+w\right) ^{\alpha +1}}{1-\beta w}\prod\limits_{j=1}^{N}%
\left( 1+x_{j}w\right) ^{-\tau _{j}}\text{,}  \label{L0}
\end{eqnarray}%
where $\alpha ,\beta ,t\in \mathbb{C}$, $w=t\left( 1+w\right) ^{\beta +1}$, $%
\tau _{n},x_{n}\in \mathbb{C}$ for $n\in 1,...,N$, the Lauricella
hypergeometric polynomials are defined as a sum%
\begin{equation*}
F_{D}^{\left( N\right) }\left( -n,\tau _{1},...,\tau _{N};\rho
;s_{1},...,s_{N}\right) =\sum_{m_{1},...,m_{N}=0}^{m_{1}+...+m_{N}\leq n}%
\frac{\left( -n\right) _{m_{1}+...+m_{N}}}{\left( \rho \right)
_{m_{1}+...+m_{N}}}\prod\limits_{j=1}^{N}\frac{\left( \tau _{j}\right)
_{m_{j}}s_{j}^{m_{j}}}{m_{j}!}\text{,}
\end{equation*}%
and $\left( a\right) _{m}$ is the Pochhammer symbol \cite{Erdelyi1} 
\begin{equation*}
\left( a\right) _{m}=\frac{\Gamma \left( a+m\right) }{\Gamma \left( a\right) 
}=\left\{ 
\begin{array}{c}
1\text{, }m=0 \\ 
\prod\limits_{l=1}^{m}\left( a+l\right) \text{, }m>0%
\end{array}%
\right. \text{.}
\end{equation*}

On setting $t=-w_{0}z$, $x_{n}w_{0}=w_{n}$, we obtain by taking $\beta =-1$, 
$\alpha =-\tau _{0}$ and $\beta =0$, $\alpha +1=-\sum_{j=0}^{N}\tau _{j}$,

\begin{eqnarray}
\prod\limits_{j=0}^{N}u_{j}^{-\tau _{j}}\left( z\right)
&=&\sum_{n=0}^{\infty }G_{n}^{\left( N\right) }\left( \tau _{j},w_{j}\right)
z^{n}\text{, }  \label{Laur1} \\
u_{j}\left( z\right) &=&1-w_{j}z\text{, }z,w_{j},\tau _{j}\in \mathbb{C}%
\text{, }w_{j}\neq w_{i}\text{,}  \notag
\end{eqnarray}%
where functions $G_{n}^{\left( N\right) }\left( \tau _{j},w_{j}\right) $ are
expressed in terms of\ the Lauricella hypergeometric polynomials: 
\begin{eqnarray*}
G_{n}^{\left( N\right) }\left( \tau _{j},w_{j}\right) &=&\frac{\left( \tau
_{0}\right) _{n}w_{0}^{n}}{n!}F_{D}^{\left( N\right) }\left( -n,\tau
_{1},...,\tau _{N};1-\tau _{0}-n;\frac{w_{1}}{w_{0}},...,\frac{w_{N}}{w_{0}}%
\right) \\
&=&\frac{\left( \sum_{j=0}^{N}\tau _{j}\right) _{n}}{w_{0}^{-n}n!}%
F_{D}^{\left( N\right) }\left( -n,\tau _{1},...,\tau
_{N};-\sum_{j=0}^{N}\tau _{j};1-\frac{w_{1}}{w_{0}},...,1-\frac{w_{N}}{w_{0}}%
\right) \text{,}
\end{eqnarray*}%
and therefore are polynomials of the degree $2N$ in the variables $\tau _{j} 
$ and $w_{j}$.

Note that the power\ series in (\ref{Laur1}) is convergent if for certain $k$%
, $\max \left( \left\vert zw_{j}\right\vert \right) <1$ as $j\leq k$ while
as $j>k$: 1) $zw_{j}=1$ and $\tau _{j}<0$; or 2) $\left\vert
zw_{j}\right\vert =1$, $\arg zw_{j}\neq 0$ and $\tau _{j}<-1$.

Simplest properties of the polynomials $G_{n}^{\left( N\right) }\left( \tau
_{j},w_{j}\right) $ follow from (\ref{Laur1}) by multiplication of power
series:%
\begin{eqnarray*}
G_{n}^{\left( N\right) }\left( \tau _{0j}+\tau _{1j},w_{j}\right)
&=&\sum_{m=0}^{n}G_{m}^{\left( N\right) }\left( \tau _{0j},w_{j}\right)
G_{n-m}^{\left( N\right) }\left( \tau _{1j},w_{j}\right) \text{,} \\
G_{n}^{\left( N\right) }\left( \tau _{j},w_{j}\right)
&=&\sum_{m=0}^{n}G_{m}^{\left( M\right) }\left( \tau _{j},w_{j}\right)
G_{n-m}^{\left( N-M-1\right) }\left( \widetilde{\tau }_{j},\widetilde{w}%
_{j}\right) \text{,} \\
\widetilde{\tau }_{j} &=&\tau _{j+M+1}\text{, \ }\widetilde{w}_{j}=w_{j+M+1}%
\text{. }
\end{eqnarray*}

Note, that by denoting 
\begin{equation*}
w_{j}^{\left( m\right) }=\left\{ 
\begin{array}{c}
\bigskip \frac{w_{j}}{w_{j}-w_{m}}\text{ if }j<m \\ 
\frac{w_{j}}{w_{j}-w_{m}}\text{ if }j\geq m%
\end{array}%
\right. \text{, \ }\tau _{j}^{\left( m\right) }=\left\{ 
\begin{array}{c}
\bigskip \tau _{j}\text{ if }j<m \\ 
\tau _{j+1}\text{ if }j\geq m%
\end{array}%
\right. \text{,}
\end{equation*}%
the left side of (\ref{Laur1}) can be rewritten in the form%
\begin{eqnarray}
\prod\limits_{j=0}^{N}u_{j}^{-\tau _{j}}\left( z\right) &=&A_{m}\left(
w_{j}\right) u_{m}^{-\tau _{m}}\left( z\right) \phi _{m}\left( u_{m}\left(
z\right) \right) \text{,}  \label{new} \\
A_{m}\left( w_{j}\right) &=&\prod\limits_{\substack{ j=0  \\ j\neq m}}%
^{N}\left( 1-\frac{w_{j}}{w_{m}}\right) ^{-\tau _{j}}\text{, }  \notag \\
\text{ }\phi _{m}\left( u_{m}\right) &=&\prod\limits_{\substack{ k=0  \\ %
k\neq m}}^{N}\left( 1-\frac{w_{k}}{w_{k}-w_{m}}u_{m}\right) ^{-\tau
_{k}}=\prod\limits_{j=0}^{N-1}\left( 1-w_{j}^{\left( m\right) }u_{m}\right)
^{-\tau _{j}^{\left( m\right) }}\text{, }  \notag
\end{eqnarray}%
in which the function $\phi _{m}\left( u_{m}\right) $ can be expanded into
the Maclaurin series in powers of $u_{m}$. By virtue of (\ref{Laur1})

\begin{equation}
\phi _{m}\left( u_{m}\right) =\sum_{r=0}^{\infty }c_{r}^{\left( m\right)
}u_{m}^{r}\text{, \ \ }c_{r}^{\left( m\right) }=G_{r}^{\left( N-1\right)
}\left( \tau _{j}^{\left( m\right) },w_{j}^{\left( m\right) }\right) \text{.}
\label{Mac}
\end{equation}

The asymptotic expansion for $G_{n}^{\left( N\right) }\left( \tau
_{j},w_{j}\right) $ as $n\rightarrow \infty $ can be readily obtained by
Darboux's method \cite{Szego} by using (\ref{new}) and (\ref{Mac}).

\begin{theorem}
Let $\pm \tau _{j}\notin \mathbb{N}_{0}$ if $0\leq j\leq N_{1}$, $\tau
_{j}=k_{j}\in \mathbb{N}$ if $N_{1}+1\leq j\leq N_{2}$ and $-\tau _{j}\in 
\mathbb{N}$ if $N_{2}+1\leq j\leq N$. As $w_{j}\neq w_{i}$ and $n\rightarrow
\infty $,%
\begin{eqnarray}
G_{n}^{\left( N\right) }\left( \tau _{j},w_{j}\right)
&=&\sum_{m=0}^{N_{1}}\sum_{r=0}^{K_{m}}A_{m}\left( w_{j}\right)
c_{r}^{\left( m\right) }\frac{\left( \tau _{m}-r\right) _{n}}{n!}\left(
w_{m}^{n}+O\left( \frac{w_{m}^{n}}{n}\right) \right)  \notag \\
&&+\sum_{m=N_{1}+1}^{N_{2}}\sum_{r=0}^{k_{m}-1}A_{m}\left( w_{j}\right)
w_{m}^{n}\frac{\left( n+r\right) !c_{k_{m}-r-1}^{\left( m\right) }}{r!n!}%
\text{,}  \label{Coef1}
\end{eqnarray}%
where $K_{m}$ are arbitrary fixed integers.
\end{theorem}

\begin{corollary}
Let $m$ be such indices that $-\tau _{m}\notin \mathbb{N}$. Denote$\ W=\max
\left( \left\vert w_{m}\right\vert \right) $ and for such indices $m$ that $%
\left\vert w_{m}\right\vert =W$ denote $T=\max \left( \func{Re}\tau
_{m}\right) $. Then, the leading term of the large order asymptotic
expansion for the polynomial $G_{n}^{\left( N\right) }\left( \tau
_{j},w_{j}\right) $ is 
\begin{equation}
G_{n}^{\left( N\right) }\left( \tau _{j},w_{j}\right) =\sum_{\func{Re}\tau
_{m}=T,\text{ }\left\vert w_{m}\right\vert =W\text{ }}\frac{A_{m}\left(
w_{j}\right) w_{m}^{n}}{n^{1-\tau _{m}}\Gamma \left( \tau _{m}\right) }
\label{L-as}
\end{equation}%
with an error whose modulus is of order $o\left( W^{n}n^{T-1}\right) $.
\end{corollary}

Below, we elaborate on several particular cases of (\ref{Laur1}) and related
polynomials.

For $N=1$ (\ref{Laur1}) turns into the well-known generation function for
the Gauss hypergeometric polynomials\ \cite{Erdelyi1}

\begin{equation}
\left( 1-z\right) ^{\tau -\rho }\left( 1-\left( 1-s\right) z\right) ^{-\tau
}=\sum_{n=0}^{\infty }g_{n}\left( \tau ,\rho ,s\right) z^{n}\text{,}
\label{gen1}
\end{equation}%
where $z,s,\tau ,\rho \in \mathbb{C}$, $\left\vert z\right\vert <1$, $%
\left\vert z\left( 1-s\right) \right\vert <1$, and%
\begin{equation*}
g_{n}\left( \tau ,\rho ,s\right) =\frac{\left( \rho -\tau \right) _{n}}{n!}%
F\left( -n,\tau ;1+\tau -\rho -n;1-s\right) =\frac{\left( \rho \right) _{n}}{%
n!}F\left( -n,\tau ;\rho ;s\right)
\end{equation*}%
are three-variable polynomials of degree $2n$ in the variables $\tau $, $%
\rho $ and $s$ whose asymptotics as $s\neq 0$ and $n\rightarrow \infty $
follows from (\ref{L-as}): if $\left\vert 1-s\right\vert <1\ $or $\left\vert
1-s\right\vert =1$, $\func{Re}\rho >2\func{Re}\tau $,%
\begin{equation*}
g_{n}\left( \tau ,\rho ,s\right) \sim \frac{s^{-\tau }}{\Gamma \left( \rho
-\tau \right) n^{1+\tau -\rho }}\text{ ;}
\end{equation*}%
if $\left\vert 1-s\right\vert =1$, $\func{Re}\rho =2\func{Re}\tau $,%
\begin{equation*}
g_{n}\left( \tau ,\rho ,s\right) \sim \frac{s^{-\tau }}{\Gamma \left( \rho
-\tau \right) n^{1+\tau -\rho }}+\frac{\left( -s\right) ^{\tau -\rho }\left(
1-s\right) ^{n+\rho -\tau }}{\Gamma \left( \tau \right) n^{1-\tau }}\text{ ;}
\end{equation*}%
if $\left\vert 1-s\right\vert >1$ or $\left\vert 1-s\right\vert =1$, $\func{%
Re}\rho <2\func{Re}\tau $,%
\begin{equation*}
g_{n}\left( \tau ,\rho ,s\right) \sim \frac{\left( -s\right) ^{\tau -\rho
}\left( 1-s\right) ^{n+\rho -\tau }}{\Gamma \left( \tau \right) n^{1-\tau }}%
\text{.}
\end{equation*}%
As $\rho =0$ and $n\geq 1$, one also can obtain

\begin{equation*}
g_{n}\left( \tau ,0,s\right) =-s\tau F\left( 1-n,1+\tau ;2;s\right) \text{.}
\end{equation*}

When $z=-u$ and $s=2$, (\ref{gen1}) becomes%
\begin{equation}
\left( 1+u\right) ^{\tau -\rho }\left( 1-u\right) ^{-\tau
}=\sum_{n=0}^{\infty }g_{n}\left( \tau ,-\rho \right) u^{n}\text{,}
\label{gen2}
\end{equation}%
where $u,\tau ,\rho \in \mathbb{C}$, $\left\vert u\right\vert <1$, and 
\begin{equation*}
g_{n}\left( \tau ,r\right) =\left( -1\right) ^{n}g_{n}\left( \tau
,-r,2\right)
\end{equation*}%
are the polynomials studied by Bateman \cite{Bateman}.

In the case $\rho =0$, $s=2$ and$\ $ $\tau =-\sigma $, (\ref{gen1}) turns
into the generating function for Mittag-Leffler polynomials%
\begin{equation}
\left( \frac{1+z}{1-z}\right) ^{\sigma }=\sum_{n=0}^{\infty }g_{n}\left(
\sigma \right) z^{n}\text{, \ }\left\vert z\right\vert <1\text{,}
\label{gen3}
\end{equation}%
\begin{equation}
\text{\ }g_{0}\left( \sigma \right) =1\text{, }g_{n}=2\sigma F\left(
1-n,1-\sigma ;2;2\right) \text{ as }n\geq 1\text{,}  \notag
\end{equation}%
\begin{equation*}
g_{n}\left( \sigma \right) \sim \frac{2^{-\sigma \mathsf{\mathtt{\mathtt{sign%
}}}\left( \func{Re}\sigma \right) }\left( \mathsf{\mathtt{\mathtt{sign}}}%
\left( -\func{Re}\sigma \right) \right) ^{n}}{n^{1-\sigma \mathsf{\mathtt{%
\mathtt{sign}}}\left( \func{Re}\sigma \right) }\Gamma \left( \sigma \mathtt{%
sign}\left( \func{Re}\sigma \right) \right) }as\ n\rightarrow \infty \text{,}
\end{equation*}%
which were introduced by Mittag-Leffler in his study of linear differential
equations \cite{Mittag,Bateman}.

In the case $\rho =2\tau $, according to (\ref{Hyper1}), 
\begin{eqnarray*}
g_{n}\left( \tau ,2\tau ,s\right) &=&\frac{\left( \tau \right) _{n}}{n!}%
F\left( -n,\tau ;1-\tau -n;1-s\right) \\
&=&\frac{\left( \tau \right) _{n}\Gamma \left( 1-\tau -n\right) \left\vert
\zeta -1\right\vert ^{\frac{\tau +n}{2}}}{2^{\tau }n!\left( \zeta +1\right)
^{\frac{n-\tau }{2}}}P_{n+\left( -\tau -n\right) }^{\tau +n}\left( \zeta
\right) \text{, }\zeta =\frac{2-s}{s}\text{,}
\end{eqnarray*}%
and then by virtue of (\ref{P-C}) we obtain by manipulations with gamma
functions 
\begin{equation*}
g_{n}\left( \tau ,2\tau ,s\right) =\frac{2^{-2n}\left( 2\tau \right) _{n}}{%
\left( \frac{1}{2}+\tau \right) _{n}}s^{n}C_{n}^{1/2-\tau -n}\left( \frac{s-2%
}{s}\right) \text{.}
\end{equation*}%
Hence, on making changes $z=e^{-u-\alpha }$ and $s=2e^{\alpha }\cosh \alpha $
the generation function (\ref{gen1}) turns into%
\begin{equation}
\frac{1}{\left( \sinh u+\sinh \alpha \right) ^{\tau }}=\sum_{n=0}^{\infty }%
\mathfrak{C}_{n}\left( \alpha ,\tau \right) e^{-\left( n+\tau \right) u}%
\text{,}  \label{sh-C}
\end{equation}
\begin{equation}
\mathfrak{C}_{n}\left( \alpha ,\tau \right) =\frac{2^{\tau -n}\left( 2\tau
\right) _{n}}{\left( \frac{1}{2}+\tau \right) _{n}}\cosh ^{n}\alpha
C_{n}^{1/2-\tau -n}\left( \tanh \alpha \right) \text{,}  \label{CC}
\end{equation}%
where $2^{-\tau }\mathfrak{C}_{n}\left( \alpha ,\tau \right) $ is a
polynomial in the parameter $\tau $.

Below, by using the differential equation for $C_{n}^{1/2-\tau -n}\left(
\tanh \alpha \right) $, we obtain a uniform asymptotic expansion as $n$ is
large.

\begin{theorem}
If $\alpha \in \left[ 0,\infty \right] $, $\tau \in \mathbb{C}$,\ and$\
n\rightarrow \infty $, then there is the uniform asymptotic formula 
\begin{equation}
C_{n}^{1/2-\tau -n}\left( \tanh \alpha \right) =\frac{\Omega _{n}\left( \tau
\right) e^{\left( n+\tau \right) \alpha }}{\cosh ^{n+\tau }\alpha }\left(
D_{n}^{+}+D_{n}^{-}e^{-2\left( n+\tau \right) \alpha }+\frac{\tau \left(
1-\tau \right) }{n+\tau }R_{n}\left( \alpha ,\tau \right) \right) \text{,}
\label{C}
\end{equation}%
where%
\begin{equation*}
\Omega _{n}\left( \tau \right) =\frac{\left( -1\right) ^{n}\Gamma \left( 
\frac{1}{2}+\tau +n\right) }{\Gamma \left( \left[ \frac{n}{2}\right]
+1\right) \Gamma \left( \frac{1}{2}+\tau +\left[ \frac{n}{2}\right] \right) }%
\text{, }
\end{equation*}%
\begin{equation*}
D_{n}^{\pm }=\frac{1+\left( -1\right) ^{n}}{4}\pm \frac{1-\left( -1\right)
^{n}}{2n+2\tau }\text{,}\ 
\end{equation*}%
and the upper estimate for $R_{n}\left( \alpha ,\tau \right) $ is%
\begin{equation*}
\left\vert R_{n}\left( \alpha ,\tau \right) \right\vert \leq \frac{2\tanh
\alpha }{1-\left\vert \frac{\tau \left( 1-\tau \right) }{n+\tau }\right\vert 
}\left\{ 
\begin{array}{c}
1\text{ \ if }n=2m\bigskip \\ 
\frac{1}{\left\vert n+\tau \right\vert }\text{ if }n=2m+1%
\end{array}%
\right. \text{, }m\in \mathbb{N}\text{.}
\end{equation*}
\end{theorem}

\begin{proof}
Gegenbauer polynomials $y\left( x\right) =C_{n}^{\nu }\left( x\right) $, $%
\nu =1/2-\tau -n$, are solutions of the initial value problem (see \cite[%
p.178, 3.15(21)]{Erdelyi1} and (\ref{Ck})) \ \ 
\begin{equation*}
(1-x^{2})\frac{d^{2}y}{dx^{2}}-(2\nu +1)x\frac{dy}{dx}+n(n+2\nu )y=0\text{,}
\end{equation*}%
\begin{equation*}
y\left( 0\right) =\Omega _{n}\left( \tau \right) \frac{1+\left( -1\right)
^{n}}{2}\text{,}
\end{equation*}%
\begin{equation*}
\frac{dy}{dx}|_{x=0}=\Omega _{n}\left( \tau \right) \left( 1-\left(
-1\right) ^{n}\right) \text{,}
\end{equation*}%
which on changing%
\begin{eqnarray}
x &=&\tanh \alpha \text{, }  \notag \\
\text{\ \ }y\left( \tanh \alpha \right) &=&\Omega _{n}\left( \tau \right) 
\frac{e^{\lambda \alpha \text{ }}}{\cosh ^{\lambda }\alpha }w\left( \alpha
\right) \text{, \ }  \label{Cnange} \\
\lambda &=&n+\tau \text{,}
\end{eqnarray}%
becomes\bigskip 
\begin{equation}
\frac{d^{2}w}{d\alpha ^{2}}+2\lambda \frac{dw}{d\alpha }-\frac{\tau \left(
1-\tau \right) }{\cosh ^{2}\alpha }w=0\text{,}  \label{dif1}
\end{equation}%
\begin{equation*}
w\left( 0\right) =f_{0}\text{,}\ \ f_{0}=\frac{1+\left( -1\right) ^{n}}{2}%
\text{,}
\end{equation*}%
\begin{equation*}
\frac{dw}{d\alpha }|_{\alpha =0}=\lambda f_{1}\text{,}\ \ f_{1}=\frac{%
1-\left( -1\right) ^{n}}{\lambda }-f_{0}\text{.}
\end{equation*}%
By using the multiplier $\exp \left( 2\lambda \alpha \right) $ and
integrating, we obtain from (\ref{dif1}) 
\begin{equation*}
\frac{dw}{d\alpha }=e^{-2\lambda \alpha \text{ }}\left( \tau \left( 1-\tau
\right) \int_{0}^{\alpha }\frac{e^{2\lambda s\text{ }}w\left( s\right) ds}{%
\cosh ^{2}s}+\lambda f_{1}\right) \text{,}
\end{equation*}%
\begin{equation*}
w\left( \alpha \right) =\tau \left( 1-\tau \right) \int_{0}^{\alpha
}e^{-2\lambda t\text{ }}\left( \int_{0}^{t}\frac{e^{2\lambda s\text{ }%
}w\left( s\right) ds}{\cosh ^{2}s}\right) dt+f_{1}\frac{1-e^{-2\lambda
\alpha \text{ }}}{2}+f_{0}\text{.}
\end{equation*}%
On integrating by parts the above equation turns into the Volterra equation
of the second kind%
\begin{equation}
w\left( \alpha \right) =\frac{\tau \left( 1-\tau \right) }{2\lambda }\mathbf{%
K}\left( w\left( s\right) \right) +w_{0}\left( \alpha \right) \text{, }
\label{Volt}
\end{equation}%
where the integral operator and free term are given by%
\begin{equation}
\mathbf{K}w\left( s\right) =\int_{0}^{\alpha }w\left( s\right) \frac{%
1-e^{-2\lambda \left( \alpha -s\right) \text{ }}}{\cosh ^{2}s}ds\text{,}
\end{equation}%
\begin{equation}
w_{0}\left( \alpha \right) =f_{1}\frac{1-e^{-2\lambda \alpha \text{ }}}{2}%
+f_{0}.  \label{w0}
\end{equation}%
The solution of (\ref{Volt}) can be presented in the form%
\begin{equation}
w\left( \alpha \right) =\frac{\tau \left( 1-\tau \right) }{\lambda }%
R_{n}\left( \alpha \right) +w_{0}\left( \alpha \right) \text{, }  \label{R}
\end{equation}%
where $R_{n}\left( \alpha \right) $ is a solution of the equation%
\begin{equation}
R_{n}\left( \alpha \right) =\frac{\tau \left( 1-\tau \right) }{2\lambda }%
\mathbf{K}\left( R_{n}\left( s\right) \right) +\mathbf{K}w_{0}\left(
s\right) \text{.}  \label{R-eq}
\end{equation}%
Combining (\ref{Cnange}), (\ref{R}), and (\ref{w0}), we arrive at (\ref{C}).

To obtain the upper estimate for $R_{n}\left( \alpha \right) $, note that%
\begin{equation*}
\left\vert w_{0}\left( s\right) \right\vert \leq \left\vert f_{1}\right\vert
+f_{0}\leq B_{n}\left( \tau \right) ,
\end{equation*}
\begin{equation}
B_{n}\left( \tau \right) =1+\left( -1\right) ^{n}+\frac{1-\left( -1\right)
^{n}}{\left\vert n+\tau \right\vert }=\left\{ 
\begin{array}{c}
2\text{ \ if }n=2m\bigskip \\ 
\frac{2}{\left\vert n+\tau \right\vert }\text{ if }n=2m+1%
\end{array}%
\right. \text{.}  \label{Bn}
\end{equation}%
Then, we get from (\ref{R-eq}) 
\begin{equation*}
\left\vert \mathbf{K}w_{0}\left( \alpha \right) \right\vert \leq
\int_{0}^{\alpha }\left\vert w_{0}\left( s\right) \right\vert \frac{%
\left\vert 1-e^{-2\lambda \left( \alpha -s\right) \text{ }}\right\vert }{%
2\cosh ^{2}s}ds\leq B_{n}\left( \tau \right) \tanh \alpha \text{,}
\end{equation*}%
\begin{equation*}
\left\vert R_{n}\left( \alpha \right) \right\vert \leq \sup_{\alpha \geq
0}\left\vert R_{n}\left( \alpha \right) \right\vert \left\vert \frac{\tau
\left( 1-\tau \right) }{\lambda }\right\vert \tanh \alpha +B_{n}\left( \tau
\right) \tanh \alpha \text{,}
\end{equation*}%
and hence the estimates%
\begin{equation*}
\sup_{\alpha \geq 0}\left\vert R_{n}\left( \alpha \right) \right\vert \leq 
\frac{B_{n}\left( \tau \right) }{1-\left\vert \frac{\tau \left( 1-\tau
\right) }{\lambda }\right\vert }=A_{n}\left( \tau \right) \text{,}
\end{equation*}%
\begin{equation*}
\left\vert R_{n}\left( \alpha \right) \right\vert \leq A_{n}\left( \tau
\right) \tanh \alpha \text{,}
\end{equation*}%
complete the proof.
\end{proof}

We also\ have

\begin{theorem}
Let $\alpha \in \left[ 0,\infty \right] $. For all $n\in \mathbb{N}_{0}$ and$%
\ \tau \in D_{\tau }\left( r\right) $, $D_{\tau }\left( r\right)
=\{\left\vert \tau -1\right\vert \leq r\}\cap \{\func{Re}\tau \leq 1\}$, $%
1<r<\sqrt{2}$, functions $e^{-\alpha n}\mathfrak{C}_{n}\left( \alpha ,\tau
\right) $ are uniformly bounded, that is,%
\begin{equation}
\left\vert \frac{2^{\tau -n\text{ }}\left( 2\tau \right) _{n}}{\left( \frac{1%
}{2}+\tau \right) _{n}}e^{-\alpha n}\cosh ^{n}\alpha C_{n}^{1/2-\tau
-n}\left( \tanh \alpha \right) \right\vert \leq \mathcal{K}_{\tau }<\infty 
\text{,}  \label{Geg-est}
\end{equation}%
where a constant $\mathcal{K}_{\tau }$ does not depend on $n$.
\end{theorem}

\begin{proof}
To prove (\ref{Geg-est}), we note that $\mathfrak{C}_{0}\left( \alpha ,\tau
\right) =2^{\tau }$, and for $n\geq 1$, 
\begin{equation*}
e^{-\alpha n}\mathfrak{C}_{n}\left( \alpha ,\tau \right) =\Omega _{n}\left(
\tau \right) \frac{2^{\tau -n\text{ }}\left( 2\tau \right) _{n}}{\left( 
\frac{1}{2}+\tau \right) _{n}}\frac{e^{\tau \alpha }}{\cosh ^{\tau }\alpha }%
w\left( \alpha \right) \text{,}
\end{equation*}%
where $\Omega _{n}\left( \tau \right) $ is given in Theorem 2 and $w\left(
\alpha \right) $ is a solution the Volterra equation (\ref{Volt}) with the
estimate%
\begin{equation*}
\sup \left\vert w\left( \alpha \right) \right\vert \leq \frac{\sup
\left\vert w_{0}\left( \alpha \right) \right\vert }{1-\left\vert \frac{\tau
\left( 1-\tau \right) }{n+\tau }\right\vert }\leq \frac{B_{n}\left( \tau
\right) }{1-\left\vert \frac{\tau \left( 1-\tau \right) }{1+\tau }%
\right\vert }
\end{equation*}%
that is valid under the condition $\left\vert \tau \left( 1-\tau \right)
\right\vert <\left\vert 1+\tau \right\vert $. \ Show that for $\tau \in
D_{\tau }\left( r\right) $ the above-mentioned condition is fulfilled. Let $%
\left\vert \tau -1\right\vert \leq r$. We write 
\begin{equation}
\left\vert \frac{\tau \left( 1-\tau \right) }{1+\tau }\right\vert ^{2}\leq
r^{2}\left\vert \frac{\tau }{1+\tau }\right\vert ^{2}<1  \label{t-ineq}
\end{equation}%
or by denoting $\func{Re}\tau =\xi $ and $\func{Im}\tau =\zeta $,%
\begin{equation}
\left( r^{2}-1\right) \xi ^{2}-2\xi +\left( r^{2}-1\right) \zeta ^{2}-1<0
\label{t0}
\end{equation}%
Note that $0\leq \zeta ^{2}\leq \zeta _{0}^{2}=r^{2}-\left( \xi -1\right)
^{2}$ for any given $\xi $. Hence, as $r>1$, the above inequality will be
fulfilled for all $\zeta $ if it is valid as $\zeta ^{2}=\zeta _{0}^{2}$ for
which (\ref{t0}) becomes 
\begin{equation*}
2\left( r^{2}-2\right) \xi +r^{2}\left( r^{2}-2\right) <0\text{,}
\end{equation*}%
and then (\ref{t-ineq}) is valid as $1<r^{2}<2\ $if%
\begin{equation*}
\func{Re}\tau >-\frac{r^{2}}{2}\text{.}
\end{equation*}%
Because the circle $\left\vert \tau -1\right\vert =r$ is on the right of the
line $\func{Re}\tau =-r^{2}/2$, we obtain that $\left\vert \tau \left(
1-\tau \right) \right\vert <\left\vert 1+\tau \right\vert $ as $\tau \in
D_{\tau }\left( r\right) $.

To prove the theorem for $n\geq 1$, $\tau \in D_{\tau }\left( r\right) $,
and $\alpha \in \left[ 0,\infty \right] $, we note, by using the Legendre
duplication formula, that%
\begin{equation*}
\left\vert \left( 2\tau \right) _{n}\frac{e^{\tau \alpha }}{\cosh ^{\tau
}\alpha }\right\vert B_{n}\left( \tau \right) \leq \mathcal{B}_{n}\left(
\tau \right) =\left\vert \frac{\Gamma \left( \tau +\left[ \frac{n}{2}\right]
+\frac{1}{2}\right) \Gamma \left( \tau +\left[ \frac{n}{2}\right] \right) }{%
2^{-\left\vert \func{Re}\tau \right\vert -n-1}\Gamma \left( \frac{1}{2}+\tau
\right) \Gamma \left( \tau \right) }\right\vert \text{.}
\end{equation*}%
Hence, as $\tau \in D_{\tau }\left( r\right) $ and $\alpha \in \left[
0,\infty \right] $, 
\begin{eqnarray}
\left\vert e^{-\alpha n}\mathfrak{C}_{n}\left( \alpha ,\tau \right)
\right\vert &\leq &\left\vert \frac{2^{\tau -n\text{ }}\Gamma \left( \frac{1%
}{2}+\tau \right) \mathcal{B}_{n}\left( \tau \right) \Omega _{n}\left( \tau
\right) }{\left( 1-\left\vert \frac{\tau \left( 1-\tau \right) }{1+\tau }%
\right\vert \right) }\right\vert  \notag \\
&\leq &\frac{2^{\left\vert \func{Re}\tau \right\vert +\func{Re}\tau +1\text{ 
}}\left\vert 1+\tau \right\vert }{\left\vert 1+\tau \right\vert -\left\vert
\tau \left( 1-\tau \right) \right\vert }\delta \left( n,\tau \right)
\label{CC1}
\end{eqnarray}%
where $\ 1-\sqrt{2}<1-r\leq \func{Re}\tau \leq 1$ and the positive numbers%
\begin{equation*}
\delta \left( n,\tau \right) =\frac{\left\vert \Gamma \left( \tau +\left[ 
\frac{n}{2}\right] \right) \right\vert }{\left\vert \Gamma \left( \tau
\right) \right\vert \Gamma \left( \left[ \frac{n}{2}\right] +1\right) }\leq
\left\{ 
\begin{array}{c}
1\text{, }n=1\bigskip \\ 
\frac{\Gamma \left( \left[ \frac{n}{2}\right] +\func{Re}\tau \right) }{%
\left\vert \Gamma \left( \tau \right) \right\vert \Gamma \left( \left[ \frac{%
n}{2}\right] +1\right) }\text{, }n\geq 2%
\end{array}%
\right. \text{, }
\end{equation*}%
constitute a bounded sequence because $\lim_{n\rightarrow \infty }\delta
\left( n,\tau \right) \ $is finite. Then, for the fixed $\tau $, $\delta
\left( n,\tau \right) \leq \max \delta \left( n,\tau \right) $, and the
right $\ $side of (\ref{CC1}) is less than a certain constant $\mathcal{K}%
_{\tau }$ which does not depend on $n$. Substituting (\ref{CC}) completes
the proof.
\end{proof}

Another special case of the generating function (\ref{Laur1}) that will be
employed in this article is

\begin{eqnarray}
\left( 1-wz\right) ^{\tau }\left( 1+\frac{z}{w}\right) ^{-\tau }\left(
1+z^{2}\right) ^{-\rho } &=&\sum_{n=0}^{\infty }\mathcal{G}_{n}\left( \tau
,\rho ,w\right) z^{n}\text{,}  \label{5aa} \\
z,w,\tau &\in &\mathbb{C}\text{, }\left\vert w\right\vert <1\text{,}
\end{eqnarray}
under one of the conditions: 1) $\left\vert z/w\right\vert <1$; 2) $%
\left\vert z/w\right\vert =1$, $\tau <0$; 3) $\left\vert z/w\right\vert =1$, 
$z/w\neq -1$, $\tau <1$.

Multiplication of power series for $\left( 1-wz\right) ^{\tau }\left( 1+%
\frac{z}{w}\right) ^{-\tau }$ and $\left( 1+z^{2}\right) ^{-\tau }$ yields a
representation of the three variable polynomials $w^{n}\mathcal{G}_{n}\left(
\tau ,\rho ,w\right) $ in the form of a sum of Gauss hypergeometric
polynomials, namely

\begin{equation*}
w^{n}\mathcal{G}_{n}\left( \tau ,\rho ,w\right) =\sum_{k=0}^{\left[ \frac{n}{%
2}\right] }\frac{\left( -1\right) ^{k}\left( \rho \right) _{k}}{k!}%
g_{n-2k}\left( \tau ,0,\frac{w^{2}+1}{w^{2}}\right) w^{n-2k}\text{,}
\end{equation*}%
and as follows from (\ref{L-as}) 
\begin{equation*}
\mathcal{G}_{n}\left( \tau ,\rho ,w\right) \sim \frac{\left( 1+w^{2}\right)
^{\tau -\rho }}{n^{1-\tau }\Gamma \left( \tau \right) \left( -w\right) ^{n}}%
\text{ as }-\tau \notin \mathbb{N}\text{, }n\rightarrow \infty \text{.}
\end{equation*}

The changes $z=i\zeta $ and $w=i\eta $ turn (\ref{5aa}) into the generating
function

\begin{equation}
\text{ }\left( 1+\eta \zeta \right) ^{\tau }\left( 1+\frac{\zeta }{\eta }%
\right) ^{-\tau }\left( 1-\zeta ^{2}\right) ^{-\rho }=\sum_{n=0}^{\infty }%
\widehat{\mathcal{G}}_{n}\left( \tau ,\rho ,\eta \right) \zeta ^{n}\text{,}
\label{gen5bb}
\end{equation}%
where $\left\vert \eta \right\vert <1$, and in addition, 1)$\ \left\vert
\zeta /\eta \right\vert <1$; or 2) $\left\vert \zeta /\eta \right\vert =1$, $%
\tau >0$; or 3) $\left\vert \zeta /\eta \right\vert =1$, $\zeta /\eta \neq
-1 $, $\tau >-1$;%
\begin{eqnarray*}
\widehat{\mathcal{G}}_{n}\left( \tau ,\rho ,\eta \right) &=&i^{n}\mathcal{G}%
_{n}\left( \tau ,\rho ,i\eta \right) =\left( -1\right) ^{n}\sum_{k=0}^{\left[
\frac{n}{2}\right] }\frac{\left( \rho \right) _{k}}{k!}g_{n-2k}\left( \tau
,0,\frac{\eta ^{2}-1}{\eta ^{2}}\right) \eta ^{n-2k}\text{,} \\
\widehat{\mathcal{G}}_{n}\left( \tau ,\rho ,w\right) &\sim &\frac{\left(
1-\eta ^{2}\right) ^{\tau -\rho }}{n^{1-\tau }\Gamma \left( \tau \right)
\left( -\eta \right) ^{n}}\text{ as }-\tau \notin \mathbb{N}\text{, }%
n\rightarrow \infty \text{.}
\end{eqnarray*}

The second family of the generating functions employed in this article is
connected with the generating function%
\begin{equation}
\left( 1+x^{\pm 1}\sqrt{1\pm z}\right) ^{-\tau }=\sum_{n=0}^{\infty }%
\mathfrak{D}_{n}^{\left( \tau \right) }\left( x^{\pm 1}\right) \frac{\left(
\mp z\right) ^{n}}{2^{n}}\text{,}  \label{sq1}
\end{equation}%
with $0<x\leq 1$, $\left\vert z\right\vert <1$, $\tau \in \mathbb{C}$, $%
\sqrt{1\pm z}>0$ as $\arg \left( 1\pm z\right) =0$, and 
\begin{equation*}
\mathfrak{D}_{n}^{\left( \tau \right) }\left( x^{\pm 1}\right) =\frac{\left(
\tau \right) _{n}x^{\pm n}}{n!\left( x^{\pm 1}+1\right) ^{n+\tau }}F\left(
1-n,n;1-n-\tau ;\frac{1+x}{2x}\right) \text{, }
\end{equation*}%
which was obtained in \cite{Malits} under some less restrictive conditions
on parameters. The functions $\mathfrak{D}_{n}^{\left( \tau \right) }\left(
x^{\pm 1}\right) $ are analytic functions of the parameter $\tau $ and can
be also expressed in terms of Jacobi polynomials of order $n$, namely%
\begin{equation*}
\mathfrak{D}_{0}^{\left( \tau \right) }\left( x^{\pm 1}\right) =\left(
x^{\pm 1}+1\right) ^{-\tau }\text{, }\mathfrak{D}_{n}^{\left( \tau \right)
}\left( x^{\pm 1}\right) =\frac{\tau x^{\pm n}P_{n-1}^{\left( n+\tau
,-n-\tau \right) }\left( x^{\mp 1}\right) }{n\left( x^{\pm 1}+1\right)
^{n+\tau }}\text{.}
\end{equation*}

As $x=1$, 
\begin{equation*}
\mathfrak{D}_{n}^{\left( \tau \right) }\left( 1\right) =\frac{\tau }{%
2^{n+\tau }n}P_{n-1}^{\left( n+\tau ,-n-\tau \right) }\left( 1\right) =\frac{%
\tau \Gamma \left( \tau +2n\right) }{2^{n+\tau }n!\Gamma \left( \tau
+n+1\right) }\text{, }
\end{equation*}%
and by simple manipulations with gamma functions we obtain the series%
\begin{eqnarray}
\left( 1+\sqrt{1-z}\right) ^{-\tau } &=&2^{-\tau }\sum_{n=0}^{\infty }\frac{%
\left( \frac{\tau }{2}\right) _{n}\left( \frac{\tau +1}{2}\right) _{n}}{%
\left( \tau +1\right) _{n}}\frac{z^{n}}{n!}\text{,}  \label{gen4} \\
\left\vert z\right\vert &\leq &1\text{, }\tau \in \mathbb{C}\text{.}  \notag
\end{eqnarray}%
For another derivation of (\ref{gen4}) see \cite[p.101, 2.8(6)]{Erdelyi1}.

A new generating function is%
\begin{equation}
\left( 1-zt\right) ^{-\rho }\left( 1+\sqrt{1-z}\right) ^{-\tau }=2^{-\tau
}\sum_{n=0}^{\infty }\mathfrak{p}_{n}\left( \rho ,\tau ,t\right) z^{n}\text{%
, }  \label{gen5}
\end{equation}%
\ where the three variable polynomials $\mathfrak{p}_{n}\left( \rho ,\tau
,t\right) $ are expressed in terms of the Clausen's generalized
hypergeometric polynomials: 
\begin{eqnarray}
\mathfrak{p}_{n}\left( \rho ,\tau ,t\right) &=&\ _{3}F_{2}\left( 
\begin{array}{c}
-n,\frac{\tau }{2},\frac{\tau +1}{2}; \\ 
\tau +1,1-\rho -n;%
\end{array}%
\frac{1}{t}\right) \frac{t^{n}\left( \rho \right) _{n}}{n!}  \notag \\
&=&\text{ }_{3}F_{2}\left( 
\begin{array}{c}
-n,-\tau -n,\rho ; \\ 
1-\frac{\tau }{2}-n,\frac{1-\tau }{2}-n;%
\end{array}%
t\right) \frac{\tau \left( \tau +n+1\right) _{n-1}}{2^{2n}n!}\text{,}
\label{pn}
\end{eqnarray}%
$\rho ,\tau ,t,z\in \mathbb{C}$, $\left\vert z\right\vert \leq 1$, and $%
\left\vert zt\right\vert <1$; or $\left\vert zt\right\vert =1$, $\rho <0$;
or $\left\vert zt\right\vert =1$, $zt\neq 1$, $\rho <1$. The above
generating function is derived with the aid of the identity

\begin{equation}
\left( a\right) _{n-k}=\frac{\left( -1\right) ^{k}\left( a\right) _{n}}{%
\left( 1-a-n\right) _{k}}  \label{ak}
\end{equation}%
by multiplying the power series (\ref{gen4}) and the binomial series for $%
\left( 1-zt\right) ^{-\rho }$.

The asymptotic behavior of $\mathfrak{p}_{n}\left( \rho ,\tau ,t\right) $ as 
$n\rightarrow \infty $ can be found in the manner of Theorem 1. In
particular, we have as $n\rightarrow \infty $,

\begin{equation}
\mathfrak{p}_{n}\left( \rho ,\tau ,1\right) =\sum_{k=0}^{K}\left( -1\right)
^{k}\left( \tau \right) _{k}\frac{\left( \rho -\frac{k}{2}\right) _{n}}{k!n!}%
+O\left( n^{\frac{K-1}{2}-\rho }\right) \text{,}  \label{as-p}
\end{equation}%
where $K$ is an arbitrary fixed integer if $2\rho \notin \mathbb{N}$ and $%
K=2\rho -1$ if $2\rho \in \mathbb{N}$.

Another generating function%
\begin{eqnarray}
\left( 1+tz\right) ^{-\nu }\left( \frac{1+\sqrt{1\pm z^{2}}}{2}\right)
^{-\mu } &=&\sum_{n=0}^{\infty }\Omega _{n}^{\pm }\left( \nu ,\mu ,t\right)
z^{n}\text{,}  \label{sq2} \\
\left\vert tz\right\vert &<&1\text{, }\left\vert z\right\vert <1\text{,} 
\notag
\end{eqnarray}%
with the three variable polynomials

\begin{equation}
\Omega _{n}^{\pm }\left( \nu ,\mu ,t\right) =\text{ }_{4}F_{3}\left( 
\begin{array}{c}
-\frac{n}{2},\frac{1-n}{2},\frac{\mu }{2},\frac{\mu +1}{2}; \\ 
\mu +1,\frac{1-\nu -n}{2},1-\frac{\nu +n}{2};%
\end{array}%
\mp \frac{1}{t^{2}}\right) \frac{\left( \nu \right) _{n}\left( -t\right) ^{n}%
}{n!}\text{,}  \label{Sigm1}
\end{equation}%
can be derived by multiplying the series (\ref{gen4}) (with $\mp z^{2}$
instead of $z$) and the binomial series for $\left( 1+tz\right) ^{-\nu }$.
Then,%
\begin{equation*}
\Omega _{n}^{\pm }\left( \nu ,\mu ,x\right) =\left( -t\right)
^{n}\sum_{k=0}^{\left[ \frac{n}{2}\right] }\frac{\left( \frac{\mu }{2}%
\right) _{k}\left( \frac{\mu +1}{2}\right) _{k}\left( \nu \right) _{n-2k}}{%
k!\left( \mu +1\right) _{k}\left( n-2k\right) !}\left( \mp 1\right)
^{k}t^{-2k}
\end{equation*}%
and on making use of identities (\ref{ak}), $\left( n-2k\right) !=n!/\left(
-n\right) _{2k}$, and $\left( a\right) _{2k}=2^{2k}\left( \frac{a}{2}\right)
_{k}\left( \frac{a+1}{2}\right) _{k}$ we arrive at (\ref{Sigm1}).

By using the Cauchy product of series (\ref{sq1}) and (\ref{sq2}), we write
the generating function (\ref{gen4})

\begin{eqnarray}
\sum_{n=0}^{\infty }\mathfrak{N}_{n,\nu ,\mu }^{\pm }\left( x\right) z^{n}
&=&\frac{\left( 1+x^{\pm 1}\sqrt{1\pm z^{2}}\right) ^{\nu }\left( 1+\sqrt{%
1\pm z^{2}}\right) ^{-\mu }}{2^{-\mu }\left( 1+\frac{1}{\sqrt{\left\vert
x^{\mp 2}-1\right\vert }}z\right) ^{\nu }}\text{,}  \label{Final-gen} \\
\text{ \ }z,\nu ,\mu &\in &\mathbb{C}\text{, \ \ }\left\vert z\right\vert <1%
\text{, \ }0<x<1\text{, \ }\frac{\left\vert z\right\vert }{\sqrt{\left\vert
x^{\mp 2}-1\right\vert }}<1\text{,}  \notag
\end{eqnarray}%
where the closed-form expression for $\mathfrak{N}_{n}^{\pm }\left( x\right) 
$ in terms of Gauss hypergeometric polynomials and generalized
hypergeometric polynomials is given by the sum%
\begin{equation*}
\mathfrak{N}_{n,\nu ,\mu }^{\pm }\left( x\right) =\sum_{n=0}^{\left[ \frac{n%
}{2}\right] }\frac{\left( \mp 1\right) ^{k}}{2^{k}}\mathfrak{D}_{k}^{\left(
-\nu \right) }\left( x^{\pm 1}\right) \Omega _{n-2k}^{\pm }\left( \nu ,\mu
,\left\vert x^{\mp 2}-1\right\vert ^{-1/2}\right) \text{.}
\end{equation*}%
One can see that $\mathfrak{N}_{n,\nu ,\mu }^{\pm }\left( x\right) $ is a
polynomial in the parameter $\mu $ and an analytic function of the parameter 
$\nu $.

The large order asymptotics of $\mathfrak{N}_{n,\nu }^{\pm }\left( x\right) $
can be again obtained by Darboux's method:%
\begin{equation*}
\mathfrak{N}_{n,\nu ,\mu }^{-}\left( x\right) \sim \frac{2^{\nu +\mu }\left(
-1\right) ^{n}\left( 1+x\right) ^{-\mu }}{\Gamma \left( \nu \right) \left(
1-x^{2}\right) ^{\frac{n}{2}}n^{1-\nu }}\left( 1+\frac{\mathfrak{U}%
_{0,n}\left( \nu ,\mu ,x\right) }{n}\right) \text{, }x>0\text{;}
\end{equation*}%
and if $0<x<2^{-1/2}$,%
\begin{equation*}
\mathfrak{N}_{n,\nu ,\mu }^{+}\left( x\right) \sim \frac{2^{\mu +\frac{1}{2}%
}\left( \mu -\nu x\right) \cos \left( \frac{\pi n}{2}+\nu \arcsin x\right) }{%
\sqrt{\pi }\left( 1-x^{2}\right) ^{-\frac{\nu }{2}}n^{3/2}}\left( 1+\frac{%
\mathfrak{U}_{1,n}\left( \nu ,\mu ,x\right) }{n}\right) \text{;}
\end{equation*}%
if $x=2^{-1/2}$,

\begin{equation*}
\mathfrak{N}_{n,\nu ,\mu }^{+}\left( x\right) \sim \frac{2^{\nu +\mu }\left(
1+\sqrt{2}\right) ^{-\mu }}{\left( -1\right) ^{n}\Gamma \left( \nu \right)
n^{1-\nu }}+\frac{\left( \sqrt{2}\mu -\nu \right) \cos \frac{\pi \left(
2n+\nu \right) }{4}}{2^{\frac{\nu }{2}-\mu }\sqrt{\pi }n^{\frac{3}{2}}}%
\left( 1+\frac{\mathfrak{U}_{2,n}\left( \nu ,\mu ,x\right) }{n}\right) \text{%
;}
\end{equation*}%
if $2^{-1/2}<x<1$,%
\begin{equation*}
\mathfrak{N}_{n,\nu ,\mu }^{+}\left( x\right) \sim \frac{\left( -1\right)
^{n}2^{\nu +\mu }x^{n+\mu }\left( 1+x\right) ^{-\mu }}{\left( 1-x^{2}\right)
^{\frac{n}{2}}\Gamma \left( \nu \right) n^{1-\nu }}\left( 1+\frac{\mathfrak{U%
}_{3,n}\left( \nu ,\mu ,x\right) }{n}\right) \text{.}
\end{equation*}%
In the foregoing formulas, $\lim_{n\rightarrow \infty }\left( n^{-1}%
\mathfrak{U}_{l,n}\right) =0$ and $\mathfrak{U}_{l,n}$ are analytic
functions of parameters $\nu $ and $\mu $.

\section{Series relating Ferrers functions $P_{\protect\nu }^{-\protect\mu %
}\left( x\right) $ and associated Legendre functions $P_{\protect\lambda }^{-%
\protect\sigma }\left( 1/x\right) $}

We commence with the simplest case that is established by deriving relations
between $I_{3}^{+}\left( \alpha ,\sigma ,\nu \right) $ and $I_{3}^{-}\left(
\alpha ,\sigma ,\nu \right) $.

\begin{theorem}
If $\mu $,$\nu \in \mathbb{C}$, there are mutually inverse series: for $x\in
\left( 2^{-1/2},1\right) $, $x=2^{-1/2}$ as $\func{Re}\left( 3\nu -\mu
\right) <-1$, and $x\in \left( 0,1\right) $ as $-\nu \in \mathbb{N}$ or $-%
\frac{\mu +\nu +1}{2}\in \mathbb{N}_{0}$, 
\begin{equation}
P_{\nu }^{-\mu }\left( x\right) =\sum_{n=0}^{\infty }\frac{\left( \frac{\mu
+\nu +1}{2}\right) _{n}\left( \nu +1\right) _{n}}{\left( -2\right)
^{-n}n!\left( 1-x^{2}\right) ^{-\frac{n}{2}}x^{\nu +n+1}}P_{\nu +n}^{-\mu
-n}\left( \frac{1}{x}\right) \text{,}  \label{rel-aa}
\end{equation}%
and for$\ x\in \left( 0,1\right) $, 
\begin{equation}
\frac{P_{\nu }^{-\mu }\left( \frac{1}{x}\right) }{x^{\nu +1}}%
=\sum_{n=0}^{\infty }\frac{\left( \frac{\mu +\nu +1}{2}\right) _{n}\left(
\nu +1\right) _{n}}{2^{-n}n!\left( 1-x^{2}\right) ^{-\frac{n}{2}}}P_{\nu
+n}^{-\mu -n}\left( x\right) \text{.}  \label{rel-ab}
\end{equation}
\end{theorem}

\begin{proof}
Assuming temporary that $\nu >-1$, $\sigma >-1$, we write (\ref{rep5}) in
the form%
\begin{equation*}
I_{3}^{+}\left( \alpha ,\sigma ,\nu \right) =\int_{\alpha }^{\infty }\frac{%
\sinh ^{\sigma }\left( t-\alpha \right) }{\sinh ^{\sigma +2\nu +2}t}\left( 1+%
\frac{1}{\sinh ^{2}t}\right) ^{-\frac{\sigma }{2}-\nu -1}dt
\end{equation*}%
and expand $\left( 1+\sinh ^{-2}t\right) ^{\varkappa }$ into a binomial
series by supposing that $\sinh \alpha >1$. Then, on integrating term by
term 
\begin{equation*}
I_{3}^{+}\left( \alpha ,\sigma ,\nu \right) =\sum_{n=0}^{\infty }\frac{%
\left( -1\right) ^{n}\left( \frac{\sigma }{2}+\nu +1\right) _{n}}{n!}%
I_{3}^{-}\left( \alpha ,\sigma ,\nu +n\right) \text{,}
\end{equation*}%
and we arrive at (\ref{rel-aa}) as $x\in \left( 2^{-1/2},1\right) $ by
invoking (\ref{rep55}), (\ref{rep66}) and using the changes $\tanh \alpha =x$
and $\sigma +\nu +1=\mu $. Due analytical continuation the relation (\ref%
{rel-aa}) holds (by virtue of the asymptotic expansion (\ref{as21})) as $%
x\in \left( 2^{-1/2},1\right) $ for any $\mu ,\nu \in \mathbb{C}$ as well as
at the point $x=2^{-1/2}\ $if $\func{Re}\left( 3\nu -\mu \right) <-1$ . As $%
-\left( \mu +\nu +1\right) /2$ or $-\nu -1$ are non-negative integers, the
series in (\ref{rel-aa}) turns into a finite sum, and the relation is valid
for $x\in \left( 0,1\right) $ by virtue of holomorphic continuation because
both sides belong to $C^{\infty }(0,1)$.

In the same manner, (\ref{rel-ab}) can be proved by writing%
\begin{equation*}
I_{3}^{-}\left( \alpha ,\sigma ,\nu \right) =\int_{\alpha }^{\infty }\frac{%
\sinh ^{\sigma }\left( t-\alpha \right) }{\cosh ^{\sigma +2\nu +2}t}\left( 1-%
\frac{1}{\cosh ^{2}t}\right) ^{-\frac{\sigma }{2}-\nu -1}dt\text{,}
\end{equation*}%
and the proof is completed.
\end{proof}

As $\nu =k+\mu $, the above theorem give us inverse series relating
Gegenbauer polynomials of reciprocal arguments. As $\nu =-2k-\mu -1$, $k\in 
\mathbb{N}_{0}$ and $n=k-r$, we obtain the finite inverse sums for even
Gegenbauer polynomials of reciprocal arguments.

\begin{corollary}
For $\mu \in \mathbb{C}$, and $k\in \mathbb{N}_{0}$, there are the mutually
inverse sums 
\begin{eqnarray*}
\frac{\left( 2k\right) !C_{2k}^{\mu +1/2}\left( x\right) }{2^{2k}k!\left(
1-x^{2}\right) ^{k}} &=&\sum_{r=0}^{k}\frac{\left( -2k-\mu \right)
_{k-r}\left( \mu +\frac{1}{2}\right) _{k-r}}{\left( k-r\right) !}\frac{%
\left( 2r\right) !C_{2r}^{\mu +k-r+1/2}\left( \frac{1}{x}\right) }{%
2^{2r}r!\left( 1/x^{2}-1\right) ^{r}}\text{,} \\
\frac{\left( 2k\right) !C_{2k}^{\mu +1/2}\left( \frac{1}{x}\right) }{%
2^{2k}k!\left( 1-1/x^{2}\right) ^{k}} &=&\sum_{r=0}^{k}\frac{\left( -2k-\mu
\right) _{k-r}\left( \mu +\frac{1}{2}\right) _{k-r}}{\left( k-r\right) !}%
\frac{\left( 2r\right) !C_{2r}^{\mu +k-r+1/2}\left( x\right) }{%
2^{2r}r!\left( x^{2}-1\right) ^{r}}\text{.}
\end{eqnarray*}
\end{corollary}

As $\nu =-k-1$ and $\mu =k-2m$, we obtain relations for the associated
Legendre polynomials of reciprocal arguments.

\begin{corollary}
For $x\in \left( -1,1\right) $ and $m,k\in \mathbb{N}_{0}$, $0\leq m\leq k$,
there are the mutually inverse sums%
\begin{equation*}
P_{k}^{2m-k}\left( x\right) =m!k!\sum_{n=0}^{m}\frac{\left( -2\right)
^{n}\left( 1-x^{2}\right) ^{\frac{n}{2}}}{\left( m-n\right) !\left(
k-n\right) !n!}x^{k-n}P_{k-n}^{2m-k-n}\left( \frac{1}{x}\right) \text{, }
\end{equation*}%
\begin{equation*}
x^{k}P_{k}^{2m-k}\left( \frac{1}{x}\right) =m!k!\sum_{n=0}^{m}\frac{%
2^{n}\left( 1-x^{2}\right) ^{\frac{n}{2}}}{\left( m-n\right) !\left(
k-n\right) !n!}P_{k-n}^{2m-k-n}\left( x\right) \text{.}
\end{equation*}%
\bigskip
\end{corollary}

Relations between $I_{1}^{+}\left( \alpha ,\mu ,\nu \right) $ and $%
I_{1}^{-}\left( \alpha ,\mu ,\nu \right) $ allow us to prove

\begin{theorem}
The mutually inverse series 
\begin{eqnarray}
\frac{P_{\nu }^{-\mu }\left( x\right) }{x^{\nu }} &=&\sum_{n=0}^{\infty
}\left( \mu -\nu \right) _{n}g_{n}\left( \nu \right) \left( \frac{1-x}{1+x}%
\right) ^{\frac{n}{2}}P_{\nu }^{-n-\mu }\left( \frac{1}{x}\right) \text{,}
\label{theor-a} \\
P_{\nu }^{-\mu }\left( \frac{1}{x}\right) &=&\sum_{n=0}^{\infty }\left( \mu
-\nu \right) _{n}g_{n}\left( -\nu \right) \left( \frac{1-x}{1+x}\right) ^{%
\frac{n}{2}}\frac{P_{\nu }^{-n-\mu }\left( x\right) }{x^{\nu }}\text{,}
\label{theor-b}
\end{eqnarray}%
are valid as $x\in \left( 0,1\right) $, $\mu $,$\nu \in \mathbb{C}$, and $%
g_{n}\left( \nu \right) $ are Mittag-Leffler polynomials.
\end{theorem}

\begin{proof}
Let $\func{Re}\mu >\func{Re}\nu >-1$ and $\alpha >0$. Rewrite $%
I_{1}^{+}\left( \alpha ,\mu ,\nu \right) $ by means of the identity $\left(
\sinh t-\sinh \alpha \right) ^{\nu }=\left( \cosh t-\cosh \alpha \right)
^{\nu }\tanh ^{-\nu }\left( \frac{t+\alpha }{2}\right) $ in the form%
\begin{equation*}
I_{1}^{+}\left( \alpha ,\mu ,\nu \right) =\int_{\alpha }^{\infty }\frac{%
e^{-t\mu }\tanh ^{-\nu }\left( \frac{t+\alpha }{2}\right) dt}{\left( \cosh
t-\cosh \alpha \right) ^{-\nu }}\text{.}
\end{equation*}%
Expanding $\tanh ^{-\nu }\left( \frac{t+\alpha }{2}\right) =\left(
1+e^{-\left( \alpha +t\right) }\right) ^{\nu }\ \left( 1-e^{-\left( \alpha
+t\right) }\right) ^{-\nu }$ into a series in powers of $e^{-t-\alpha }$
with coefficients expressed in the terms of Mittag-Leffler polynomials (see (%
\ref{gen3})), we obtain 
\begin{equation*}
I_{1}^{+}\left( \alpha ,\mu ,\nu \right) =\int_{\alpha }^{\infty }\left(
\sum_{n=0}^{\infty }g_{n}\left( \nu \right) \frac{e^{-\alpha n}e^{-t\left(
\mu +n\right) }}{\left( \cosh t-\cosh \alpha \right) ^{-\nu }}\right) dt%
\text{.}
\end{equation*}%
where power series converges absolutely and uniformly on the interval $t\geq
0$. Because terms of an absolutely and uniformly convergent power series are
uniformly bounded and $e^{-t\mu }\left( \cosh t-\cosh \alpha \right) ^{-\nu
}\in L\left( \alpha ,\infty \right) $, the dominated convergence theorem
allows us integration term by term leading to 
\begin{equation*}
I_{1}^{+}\left( \alpha ,\mu ,\nu \right) =\sum_{n=0}^{\infty }g_{n}\left(
\nu \right) e^{-\alpha n}I_{1}^{-}\left( \alpha ,\mu +n,\nu \right) \text{,}
\end{equation*}%
Hence, on employing (\ref{rep11}) and (\ref{rep33})%
\begin{equation}
\frac{P_{\nu }^{-\mu }\left( \tanh \alpha \right) }{\tanh ^{\nu }\alpha }%
=\sum_{n=0}^{\infty }\left( \mu -\nu \right) _{n}g_{n}\left( \nu \right)
e^{-\alpha n}P_{\nu }^{-n-\mu }\left( \coth \alpha \right) \text{,}
\label{theor1}
\end{equation}%
and the change $\tanh \alpha =x$ shows that (\ref{theor-a}) is valid as $%
\func{Re}\mu >\func{Re}\nu >-1$. \ The asymptotic formula (\ref{as11})
manifests that for any $\mu ,\nu \in \mathbb{C}$ and $\alpha >0$, $%
\left\vert P_{\nu }^{-n-\mu }\left( \coth \alpha \right) \left( \mu -\nu
\right) _{n}\right\vert \rightarrow 0$ as $n\rightarrow \infty $, and by
virtue of asymptotics of Mittag-Leffler polynomials the series in (\ref%
{theor1}) converges absolutely and uniformly. Note that the left side of (%
\ref{theor1}) is an analytic function of the parameters $\mu $ and $\nu $
while terms of the series in the right side also are analytic functions of $%
\mu $ and $\nu $. Then, (\ref{theor-a}) is proved by analytic continuation.

The inverse relation (\ref{theor-b}) can be established in the similar way
by using%
\begin{equation*}
I_{1}^{-}\left( \alpha ,\mu ,\nu \right) =\int_{\alpha }^{\infty }\frac{%
e^{-t\mu }\tanh ^{\nu }\left( \frac{t+\alpha }{2}\right) dt}{\left( \sinh
t-\sinh \alpha \right) ^{-\nu }}\text{,}
\end{equation*}%
(\ref{gen3}) and (\ref{as12}).
\end{proof}

As $\mu -\nu =-k$, $k\in \mathbb{N}_{0}$, and $\mu =\lambda -1/2$, (\ref%
{theor-a}) and (\ref{theor-b}) turn into new relations for Gegenbauer
polynomials which by making the change $n=k-m$ are written as follows.

\begin{corollary}
For $\lambda \in \mathbb{C}$, there are the mutually inverse sums%
\begin{eqnarray*}
\frac{C_{k}^{\lambda }\left( x\right) }{2^{k}\left( x-1\right) ^{k}}
&=&\sum_{m=0}^{k}\frac{g_{k-m}\left( k+\lambda -\frac{1}{2}\right) \left(
\lambda \right) _{k-m}}{\left( k+2\lambda \right) _{k-m}}\frac{%
C_{m}^{k-m+\lambda }\left( \frac{1}{x}\right) }{2^{m}\left( 1-\frac{1}{x}%
\right) ^{m}}\text{,} \\
\frac{C_{k}^{\lambda }\left( \frac{1}{x}\right) }{2^{k}\left( 1-\frac{1}{x}%
\right) ^{k}} &=&\sum_{m=0}^{k}\frac{g_{k-m}\left( \frac{1}{2}-k-\lambda
\right) \left( \lambda \right) _{k-m}}{\left( k+2\lambda \right) _{k-m}}%
\frac{C_{m}^{k-m+\lambda }\left( x\right) }{2^{m}\left( x-1\right) ^{m}}%
\text{.}
\end{eqnarray*}
\end{corollary}

In a similar manner, relations between $I_{3}^{\pm }\left( \alpha ,\nu ,\mu
-\nu -1\right) $ and $I_{2}^{\mp }\left( \alpha ,\mu ,\nu \right) $ can be
established by using two pairs of identities%
\begin{eqnarray*}
\frac{\sinh ^{\nu }\left( t-\alpha \right) }{\cosh ^{2\mu -\nu }t} &=&\frac{%
2^{2\mu }e^{-2\mu t}}{\sinh ^{-\nu }\left( t-\alpha \right) \sinh ^{-\nu }t}%
\frac{\left( 1+e^{-2t}\right) ^{\nu -2\mu }}{\left( 1-e^{-2t}\right) ^{\nu }}%
\text{, } \\
\frac{e^{-2t\mu }\sinh ^{\nu }t}{\sinh ^{-\nu }\left( t-\alpha \right) } &=&%
\frac{\sinh ^{\nu }\left( t-\alpha \right) }{\cosh ^{\nu +2\left( \mu -\nu
-1\right) +2}t}\frac{\left( 1-\frac{1}{\cosh ^{2}t}\right) ^{\frac{\nu }{2}}%
}{\left( 1+\sqrt{1-\frac{1}{\cosh ^{2}t}}\right) ^{2\mu }}\text{, }t>0\text{;%
}
\end{eqnarray*}%
and%
\begin{eqnarray*}
\frac{\sinh ^{\nu }\left( t-\alpha \right) }{\sinh ^{2\mu -\nu }t} &=&\frac{%
2^{2\mu }e^{-2\mu t}}{\sinh ^{-\nu }\left( t-\alpha \right) \cosh ^{-\nu }t}%
\frac{\left( 1-e^{-2t}\right) ^{\nu -2\mu }}{\left( 1+e^{-2t}\right) ^{\nu }}%
\text{, } \\
\frac{e^{-2t\mu }\cosh ^{\nu }t}{\sinh ^{-\nu }\left( t-\alpha \right) } &=&%
\frac{\sinh ^{\nu }\left( t-\alpha \right) }{\sinh ^{\nu +2\left( \mu -\nu
-1\right) +2}t}\frac{\left( 1+\frac{1}{\sinh ^{2}t}\right) ^{\frac{\nu }{2}}%
}{\left( 1+\sqrt{1+\frac{1}{\sinh ^{2}t}}\right) ^{2\mu }}\text{, }t>0\text{;%
}
\end{eqnarray*}%
and invoking the generating functions (\ref{gen2}) and (\ref{gen5}). On
employing (\ref{rep5}), (\ref{rep6}), (\ref{rep2}) and (\ref{rep4}) we
derive the following result.

\begin{theorem}
There are two pairs of mutually inverse series: 1. For $x\in \left(
2^{-1/2},1\right) $, $x=2^{-1/2}$ as $\func{Re}\nu >-2/3$, $x\in \left(
0,1\right) $ as $\nu -\mu \in \mathbb{N}_{0}$,%
\begin{equation}
\frac{P_{\nu }^{-\mu }\left( x\right) }{\left( 1+x\right) ^{\mu }}%
=\sum_{n=0}^{\infty }\frac{\left( -1\right) ^{n}\left( \mu -\nu \right) _{n}%
\mathfrak{p}_{n}\left( -\frac{\nu }{2},2\mu ,1\right) }{2^{\mu -n}x^{n+\mu
-\nu }\left( 1-x^{2}\right) ^{-\frac{n}{2}}}P_{\nu -\mu -n}^{-\mu -n}\left( 
\frac{1}{x}\right) \text{, }  \label{I2- I3A}
\end{equation}%
and for $x\in \left( 0,1\right) $,%
\begin{equation}
\frac{P_{\nu -\mu }^{-\mu }\left( \frac{1}{x}\right) }{2^{\mu }x^{\mu -\nu }}%
=\sum_{n=0}^{\infty }\left( -1\right) ^{n}\left( \mu -\nu \right)
_{n}g_{n}\left( \nu ,-2\mu \right) \frac{\left( 1-x\right) ^{\frac{n}{2}%
}P_{\nu }^{-\mu -n}\left( x\right) }{\left( 1+x\right) ^{\frac{n}{2}+\mu }}%
\text{.}  \label{I2-I3B}
\end{equation}%
2. For $x\in \left( 0,1\right) $,%
\begin{eqnarray}
\frac{x^{\nu }P_{\nu }^{-\mu }\left( \frac{1}{x}\right) }{\left( 1+x\right)
^{\mu }} &=&\sum_{n=0}^{\infty }\frac{\left( \mu -\nu \right) _{n}\mathfrak{p%
}_{n}\left( -\frac{\nu }{2},2\mu ,1\right) }{2^{\mu -n}\left( 1-x^{2}\right)
^{-\frac{n}{2}}}P_{\nu -\mu -n}^{-n-\mu }\left( x\right) \text{,}  \label{AA}
\\
\frac{P_{\nu -\mu }^{-\mu }\left( x\right) }{2^{\mu }x^{\nu }}
&=&\sum_{n=0}^{\infty }\left( \mu -\nu \right) _{n}g_{n}\left( \nu ,-2\mu
\right) \frac{\left( 1-x\right) ^{\frac{n}{2}}P_{\nu }^{-n-\mu }\left( \frac{%
1}{x}\right) }{\left( 1+x\right) ^{\frac{n}{2}+\mu }}\text{.}  \label{AA0}
\end{eqnarray}
\end{theorem}

As $\nu =k+\mu $, $k\in \mathbb{N}_{0}$, the above theorem yields finite
sums relating Gegenbauer and Jacobi polynomials. As $\mu =\pm m$, $m\leq k$,
and $m\in \mathbb{N}_{0}$, some of them turn into relations for the
associated Legendre polynomials. We obtain from (\ref{I2-I3B}) and (\ref{AA0}%
)

\begin{corollary}
For $x\in \left( -1,1\right) $ there are two pairs of the mutually inverse
relations:%
\begin{eqnarray*}
\frac{P_{k}^{-m}\left( x\right) }{2^{m}k!x^{k+m}} &=&\sum_{n=0}^{k}\frac{%
\left( -1\right) ^{n}g_{n}\left( k+m,-2m\right) }{\left( k-n\right) !}\frac{%
\left( 1-x\right) ^{\frac{n}{2}}P_{k+m}^{-n-m}\left( \frac{1}{x}\right) }{%
\left( 1+x\right) ^{\frac{n}{2}+m}}\text{,} \\
\frac{P_{k}^{-m}\left( \frac{1}{x}\right) }{2^{m}k!x^{-k}} &=&\sum_{n=0}^{k}%
\frac{g_{n}\left( k+m,-2m\right) }{\left( k-n\right) !}\frac{\left(
1-x\right) ^{\frac{n}{2}}P_{k+m}^{-n-m}\left( x\right) }{\left( 1+x\right) ^{%
\frac{n}{2}+m}}\text{,}
\end{eqnarray*}%
and%
\begin{eqnarray*}
\frac{P_{k}^{m}\left( x\right) }{2^{-m}k!x^{k-m}} &=&\sum_{n=\max \left(
0,2m-k\right) }^{k}\frac{\left( -1\right) ^{n}g_{n}\left( k-m,2m\right) }{%
\left( k-n\right) !}\frac{\left( 1-x\right) ^{\frac{n}{2}}P_{k-m}^{m-n}%
\left( \frac{1}{x}\right) }{\left( 1+x\right) ^{\frac{n}{2}-m}}\text{, } \\
\frac{P_{k}^{m}\left( \frac{1}{x}\right) }{2^{-m}k!x^{-k}} &=&\sum_{n=\max
\left( 0,2m-k\right) }^{k}\frac{g_{n}\left( k-m,2m\right) }{\left(
k-n\right) !}\frac{\left( 1-x\right) ^{\frac{n}{2}}P_{k-m}^{m-n}\left(
x\right) }{\left( 1+x\right) ^{\frac{n}{2}-m}}\text{.}
\end{eqnarray*}
\end{corollary}

Another result for the associated Legendre polynomials follows from (\ref%
{I2- I3A}), (\ref{AA}) and (\ref{Pn}) as $\mu =-m$ and $\nu =k-m$.

\begin{corollary}
As $k/2<m\leq k$, $k.m\in \mathbb{N}_{0}$,\ and $x>0$, there is the
connection formula for the associated Legendre polynomials%
\begin{equation*}
\sum_{n=0}^{k}\frac{2^{n}\mathfrak{p}_{n}\left( \frac{m-k}{2},-2m,1\right) }{%
\left( k-n\right) !}\left( x-1\right) ^{n}\left\vert \frac{1+x}{1-x}%
\right\vert ^{\frac{n}{2}}P_{k-n}^{m-n}\left( x\right) =0\text{.}
\end{equation*}
\end{corollary}

The identities%
\begin{eqnarray*}
\frac{e^{-\mu t}}{\left( \cosh t-\cosh \alpha \right) ^{-\nu }} &=&\frac{%
e^{-\left( \mu +\nu \right) t}\cosh ^{\nu }t}{\sinh ^{-\nu }\left( t-\alpha
\right) }\frac{2^{\nu }e^{\alpha \nu }\left( 1-e^{-\alpha }e^{-t}\right)
^{\nu }}{\left( 1+e^{\alpha }e^{-t}\right) ^{\nu }\left( 1+e^{-2t}\right)
^{\nu }}\text{,} \\
\frac{e^{-\mu t}}{\left( \sinh t-\sinh \alpha \right) ^{-\nu }} &=&\frac{%
e^{-\left( \mu +\nu \right) t}\sinh ^{\nu }t}{\sinh ^{-\nu }\left( t-\alpha
\right) }\frac{2^{\nu }e^{\alpha \nu }\left( 1+e^{-\alpha }e^{-t}\right)
^{\nu }}{\left( 1+e^{\alpha }e^{-t}\right) ^{\nu }\left( 1-e^{-2t}\right)
^{\nu }}\text{,}
\end{eqnarray*}%
together with the integrals $I_{1}^{\pm }\left( \alpha ,\mu ,\nu \right) $, $%
I_{2}^{\mp }\left( \alpha ,\left( \mu +\nu \right) /2,\nu \right) ,$ and the
generation functions (\ref{5aa}) and (\ref{gen5bb}) allow us to prove

\begin{theorem}
For $x\in \left( 0,1\right) \mathbb{\ }$and $\mu ,\nu \in \mathbb{C}$ there
are two pairs of the mutually inverse relations:%
\begin{equation}
\frac{P_{\nu }^{-\mu }\left( \frac{1}{x}\right) }{\sqrt{\pi }2^{\nu -\mu
}x^{-\nu }}=\sum_{n=0}^{\infty }\frac{\left( \mu -\nu \right) _{n}\mathcal{G}%
_{n}\left( \nu ,\nu ,\sqrt{\frac{1-x}{1+x}}\right) }{2^{n}\Gamma \left( 
\frac{n+\mu -\nu +1}{2}\right) }\left( \frac{1-x}{1+x}\right) ^{\frac{n+\mu
-\nu }{4}}P_{\nu }^{-\frac{\mu +\nu +n}{2}}\left( x\right) \text{;}
\label{q1}
\end{equation}%
\begin{equation}
\left( \frac{1-x}{1+x}\right) ^{\frac{\mu -\nu }{4}}\frac{P_{\nu }^{-\frac{%
\mu +\nu }{2}}\left( x\right) }{\Gamma \left( \frac{\mu -\nu +1}{2}\right) }%
=\sum_{n=0}^{\infty }\frac{\left( \mu -\nu \right) _{n}\mathcal{G}_{n}\left(
-\nu ,-\nu ,\sqrt{\frac{1-x}{1+x}}\right) }{\sqrt{\pi }2^{\nu -\mu }x^{-\nu }%
}P_{\nu }^{-\mu -n}\left( \frac{1}{x}\right) \text{, }  \label{q2}
\end{equation}%
\begin{equation*}
\func{Re}\nu >-1\text{ or }\nu -\mu \in \mathbb{N}_{0}\text{;}
\end{equation*}%
and%
\begin{equation}
\frac{P_{\nu }^{-\mu }\left( x\right) }{\sqrt{\pi }2^{\nu -\mu }x^{\nu }}%
=\sum_{n=0}^{\infty }\frac{\left( \mu -\nu \right) _{n}\widehat{\mathcal{G}}%
_{n}\left( \nu ,\nu ,\sqrt{\frac{1-x}{1+x}}\right) }{2^{n}\Gamma \left( 
\frac{n+\mu -\nu +1}{2}\right) }\left( \frac{1-x}{1+x}\right) ^{\frac{n+\mu
-\nu }{4}}P_{\nu }^{-\frac{\mu +\nu +n}{2}}\left( \frac{1}{x}\right) \text{;}
\label{q3}
\end{equation}%
\begin{equation}
\left( \frac{1-x}{1+x}\right) ^{\frac{\mu -\nu }{4}}\frac{P_{\nu }^{-\frac{%
\mu +\nu }{2}}\left( \frac{1}{x}\right) }{\Gamma \left( \frac{\mu -\nu +1}{2}%
\right) }=\sum_{n=0}^{\infty }\frac{\left( \mu -\nu \right) _{n}\widehat{%
\mathcal{G}}_{n}\left( -\nu ,-\nu ,\sqrt{\frac{1-x}{1+x}}\right) }{\sqrt{\pi 
}2^{\nu -\mu }x^{\nu }}P_{\nu }^{-\mu -n}\left( x\right) \text{, }
\label{q4}
\end{equation}%
\begin{equation*}
\func{Re}\nu >-1\text{ or }\nu -\mu \in \mathbb{N}_{0}\text{.}
\end{equation*}
\end{theorem}

Setting $\nu =k+\mu $, $\left[ k/2\right] -\left[ n/2\right] =r$ in (\ref{q1}%
) and (\ref{q3}), as well as $\nu =2k+\mu $ in (\ref{q2}) and (\ref{q4}), on
denoting $\mu =\lambda -1/2$ we obtain by simple manipulations inverse sums
relating Gegenbauer polynomials.

\begin{corollary}
As $\lambda \in \mathbb{C}$, $k\in \mathbb{N}_{0}$, and $x\in \left(
-1,1\right) $, there are the mutually inverse sums 
\begin{equation*}
\frac{C_{k}^{\lambda }\left( \frac{1}{x}\right) }{\left( 1-x^{2}\right) ^{%
\frac{k}{2}}x^{-k}}=\sum_{r=0}^{\left[ \frac{k}{2}\right] }\frac{\left(
\lambda \right) _{k-r}\mathcal{G}_{k-2r}\left( k+\lambda -\frac{1}{2}%
,k+\lambda -\frac{1}{2},\sqrt{\frac{1-x}{1+x}}\right) }{\left( -2\right)
^{r+k}\left( 2\lambda +k\right) _{k-r}\left( 1-x\right) ^{r}}%
C_{r}^{k-r+\lambda }\left( x\right) \text{,}
\end{equation*}%
\begin{equation*}
\frac{C_{k}^{k+\lambda }\left( x\right) }{\left( 1-x\right) ^{k}}%
=\sum_{n=0}^{2k}\frac{\mathcal{G}_{2k-r}\left( \frac{1}{2}-2k-\lambda ,\frac{%
1}{2}-2k-\lambda ,\sqrt{\frac{1-x}{1+x}}\right) }{\left( -2\right) ^{r-k}%
\mathfrak{Y}_{\lambda }\left( k,r\right) \left( 1-x^{2}\right) ^{\frac{r}{2}%
}x^{-r}}C_{r}^{2k-r+\lambda }\left( \frac{1}{x}\right) \text{,}
\end{equation*}%
and%
\begin{eqnarray*}
\frac{C_{k}^{\lambda }\left( x\right) }{\left( 1-x^{2}\right) ^{\frac{k}{2}}}
&=&\sum_{r=0}^{\left[ \frac{k}{2}\right] }\frac{\left( \lambda \right) _{k-r}%
\widehat{\mathcal{G}}_{k-2r}\left( k+\lambda -\frac{1}{2},k+\lambda -\frac{1%
}{2},\sqrt{\frac{1-x}{1+x}}\right) }{\left( -2\right) ^{k+r}\Gamma \left(
2\lambda +k\right) _{k-r}\left( 1-x\right) ^{r}x^{-r}}C_{r}^{k-r+\lambda
}\left( \frac{1}{x}\right) \text{,} \\
\frac{C_{k}^{k+\lambda }\left( \frac{1}{x}\right) }{\left( 1-x\right)
^{k}x^{-k}} &=&\sum_{r=0}^{2k}\frac{\widehat{\mathcal{G}}_{n}\left( \frac{1}{%
2}-2k-\lambda ,\frac{1}{2}-2k-\lambda ,\sqrt{\frac{1-x}{1+x}}\right) }{%
\left( -2\right) ^{r-k}\mathfrak{Y}_{\lambda }\left( k,r\right) \left(
1-x^{2}\right) ^{\frac{r}{2}}}C_{r}^{2k-r+\lambda }\left( x\right) \text{,}
\end{eqnarray*}%
where 
\begin{equation*}
\mathfrak{Y}_{\lambda }\left( k,r\right) =\frac{\Gamma \left( \lambda
+k\right) \Gamma \left( 2\lambda +4k-r\right) }{\Gamma \left( \lambda
+2k-r\right) \Gamma \left( 2\lambda +3k\right) }\text{.}
\end{equation*}
\end{corollary}

Two extra relations arise by setting $\mu =\lambda -1/2$ and $\nu =2k+1+\mu $%
, $k\in \mathbb{N}_{0}$, in (\ref{q2}) and (\ref{q4}).

\begin{corollary}
The connecting formulas%
\begin{eqnarray}
\sum_{n=0}^{2k+1}\frac{\left( \lambda \right) _{n}\mathcal{G}_{n}\left(
-2k-\lambda -\frac{1}{2},-2k-\lambda -\frac{1}{2},\sqrt{\frac{1-x}{1+x}}%
\right) }{\left( 2\lambda +2k+1\right) _{n}\left( -2\right) ^{-n}\left(
1-x^{2}\right) ^{-\frac{n}{2}}x^{n}}C_{2k-n+1}^{\lambda +n}\left( \frac{1}{x}%
\right) &=&0\text{,}  \label{GG1} \\
\sum_{n=0}^{2k+1}\frac{\left( \lambda \right) _{n}\widehat{\mathcal{G}}%
_{n}\left( -2k-\lambda -\frac{1}{2},-2k-\lambda -\frac{1}{2},\sqrt{\frac{1-x%
}{1+x}}\right) }{\left( \lambda +2k+1\right) _{n}\left( -2\right)
^{-n}\left( 1-x^{2}\right) ^{-\frac{n}{2}}}C_{2k-n+1}^{\lambda +n}\left(
x\right) &=&0\text{,}  \label{GG2}
\end{eqnarray}%
are valid for $x\in \left( 0,1\right) $ and $\lambda \in \mathbb{C}$.
\end{corollary}

Additional inverse series can be found from two pairs of identities allowing
us to establish relations between $I_{3}^{\pm }\left( \alpha ,\nu ,\left(
\mu -\nu -2\right) /2\right) $ and $I_{1}^{\mp }\left( \alpha ,\mu ,\nu
\right) $: 
\begin{eqnarray*}
\frac{e^{-t\mu }}{\left( \cosh t-\cosh \alpha \right) ^{-\nu }} &=&\frac{%
\sinh ^{\nu }\left( t-\alpha \right) }{\cosh ^{\nu +2\frac{\mu -\nu -2}{2}%
+2}t}\frac{\left( 1+\coth \alpha \sqrt{1-\frac{1}{\cosh ^{2}t}}\right) ^{\nu
}\sinh ^{\nu }\alpha }{\left( 1+\sqrt{1-\frac{1}{\cosh ^{2}t}}\right) ^{\mu
}\left( 1+\frac{\cosh \alpha }{\cosh t}\right) ^{\nu }}\text{,} \\
\frac{\sinh ^{^{\nu }}\left( t-\alpha \right) }{\cosh ^{\nu +2\frac{\mu -\nu
-2}{2}+2}t} &=&\frac{e^{-\mu t}}{\left( \cosh t-\cosh \alpha \right) ^{-\nu }%
}\frac{2^{\mu }}{e^{\nu \alpha }}\frac{\left( 1+e^{\alpha }e^{-t}\right)
^{\nu }\left( 1-e^{-\alpha }e^{-t}\right) ^{-\nu }}{\left( 1+e^{-2t}\right)
^{\mu }}\text{;}
\end{eqnarray*}%
and%
\begin{eqnarray*}
\frac{e^{-t\mu }}{\left( \sinh t-\sinh \alpha \right) ^{-\nu }} &=&\frac{%
\sinh ^{\nu }\left( t-\alpha \right) }{\sinh ^{\nu +2\frac{\mu -\nu -2}{2}%
+2}t}\frac{\left( 1+\tanh \alpha \sqrt{1+\frac{1}{\sinh ^{2}t}}\right) ^{\nu
}\cosh ^{-\nu }\alpha }{\left( 1+\sqrt{1+\frac{1}{\sinh ^{2}t}}\right) ^{\mu
}\left( 1+\frac{\sinh \alpha }{\sinh t}\right) ^{\nu }}\text{,} \\
\frac{\sinh ^{\nu }\left( t-\alpha \right) }{\sinh ^{\nu +2\frac{\mu -\nu -2%
}{2}+2}t} &=&\frac{e^{-\mu t}}{\left( \sinh t-\sinh \alpha \right) ^{-\nu }}%
\frac{2^{\mu }}{e^{\nu \alpha }}\frac{\left( 1+e^{\alpha }e^{-t}\right)
^{\nu }\left( 1+e^{-\alpha }e^{-t}\right) ^{-\nu }}{\left( 1-e^{-2t}\right)
^{\mu }}\text{.}
\end{eqnarray*}%
Then, the usage of the generation functions (\ref{Final-gen}), (\ref{5aa}),
and (\ref{gen5bb}) as well as analytical continuation lead to

\begin{theorem}
Let $\mu ,\nu \in \mathbb{C}$. Then, two pairs of the mutually inverse
series are valid: 1. For $0<x<1$, $\func{Re}\nu >-1$ or $\nu -\mu \in 
\mathbb{N}_{0}$, 
\begin{equation}
\frac{\left( 1-x^{2}\right) ^{\frac{\mu +\nu }{4}}P_{\frac{\mu -\nu -2}{2}%
}^{-\frac{\mu +\nu }{2}}\left( x\right) }{2^{\frac{3\mu -\nu }{2}}x^{^{\nu
}}\Gamma \left( \frac{\mu -\nu +1}{2}\right) }=\sum_{n=0}^{\infty }\frac{%
\left( \mu -\nu \right) _{n}\mathcal{G}_{n}\left( -\nu ,\mu ,\sqrt{\frac{1-x%
}{1+x}}\right) }{\sqrt{\pi }\left( \frac{1-x}{1+x}\right) ^{-\frac{\nu }{2}}}%
P_{\nu }^{-\mu -n}\left( \frac{1}{x}\right) \text{,}  \label{G1}
\end{equation}%
and for $0<x<1$,%
\begin{equation}
\frac{P_{\nu }^{-\mu }\left( \frac{1}{x}\right) }{\sqrt{\pi }}%
=\sum_{n=0}^{\infty }\left( \mu -\nu \right) _{n}\mathfrak{N}_{n,\nu ,\mu
}^{-}\left( x\right) \frac{\left( 1-x^{2}\right) ^{\frac{\mu -\nu +n}{4}}P_{%
\frac{\mu -\nu -2+n}{2}}^{-\frac{\mu +\nu +n}{2}}\left( x\right) }{2^{\frac{%
3\mu -\nu +n}{2}}\Gamma \left( \frac{\mu -\nu +n+1}{2}\right) }\text{.}
\label{R1}
\end{equation}%
2. For $x\in \left( 2^{-1/2},1\right) $, $x=2^{-1/2}$ as $\func{Re}\nu <2$, $%
x\in \left( 0,1\right) $ as $\nu -\mu \in \mathbb{N}_{0}$, 
\begin{equation}
\frac{P_{\nu }^{-\mu }\left( x\right) }{\sqrt{\pi }}=\sum_{n=0}^{\infty
}\left( \mu -\nu \right) _{n}\mathfrak{N}_{n,\nu ,\mu }^{+}\left( x\right) 
\frac{\left( 1-x^{2}\right) ^{\frac{\mu -\nu +n}{4}}P_{\frac{\mu -\nu +n-2}{2%
}}^{-\frac{\mu +\nu +n}{2}}\left( \frac{1}{x}\right) }{2^{\frac{3\mu -\nu +n%
}{2}}\Gamma \left( \frac{\mu -\nu +n+1}{2}\right) x^{\frac{\mu -\nu +n}{2}}}%
\text{, }  \label{R2}
\end{equation}%
and for $\ 0<x<1$, $\func{Re}\nu >-1$,%
\begin{equation}
\frac{\left( 1-x^{2}\right) ^{\frac{\mu +\nu }{4}}P_{\frac{\mu -\nu -2}{2}%
}^{-\frac{\mu +\nu }{2}}\left( \frac{1}{x}\right) }{2^{\frac{3\mu -\nu }{2}%
}x^{\frac{\mu -\nu }{2}}\Gamma \left( \frac{\mu -\nu +1}{2}\right) }%
=\sum_{n=0}^{\infty }\frac{\left( \mu -\nu \right) _{n}\widehat{\mathcal{G}}%
_{n}\left( -\nu ,\mu ,\sqrt{\frac{1-x}{1+x}}\right) }{\sqrt{\pi }\left( 
\frac{1-x}{1+x}\right) ^{-\frac{\nu }{2}}}P_{\nu }^{-\mu -n}\left( x\right) 
\text{.}  \label{G2}
\end{equation}
\end{theorem}

Setting $\nu =2k+\mu +1$ and $\mu =\lambda -1/2$ into (\ref{G1}), and (\ref%
{G2}) yields

\begin{corollary}
If $\lambda $ $\in \mathbb{C}$, $k\in \mathbb{N}_{0}$, and $x\in \left(
0,1\right) $, then the connecting formulas%
\begin{eqnarray*}
\sum_{n=0}^{2k+1}\frac{\left( \lambda \right) _{n}\mathcal{G}_{n}\left(
-2k-\lambda -\frac{1}{2},\lambda -\frac{1}{2},\sqrt{\frac{1-x}{1+x}}\right) 
}{\left( 2\lambda +2k+1\right) _{n}\left( -2\right) ^{-n}\left(
1-x^{2}\right) ^{-\frac{n}{2}}x^{n}}C_{2k-n+1}^{\lambda +n}\left( \frac{1}{x}%
\right) &=&0\text{,} \\
\sum_{n=0}^{2k+1}\frac{\left( \lambda \right) _{n}\widehat{\mathcal{G}}%
_{n}\left( -2k-\lambda -\frac{1}{2},\lambda -\frac{1}{2},\sqrt{\frac{1-x}{1+x%
}}\right) }{\left( 2\lambda +2k+1\right) _{n}\left( -2\right) ^{-n}\left(
1-x^{2}\right) ^{-\frac{n}{2}}}C_{2k-n+1}^{\lambda +n}\left( x\right) &=&0%
\text{,}
\end{eqnarray*}%
are valid.
\end{corollary}

The above formulas are rather unexpected because they are very similar to
formulas (\ref{GG1}), and (\ref{GG2}) but different.

Setting $\nu =k\pm m$ and $\mu =-k\pm m$, leads to two new mutually inverse
relations for associated Legendre polynomials.

\begin{corollary}
If $m,k\in \mathbb{N}_{0}$, $m\leq k$, and $x\in \left( -1,1\right) $, then 
\begin{eqnarray*}
\frac{P_{k}^{\mp m}\left( x\right) }{2^{\pm m}} &=&\frac{k!\left( 1-x\right)
^{\frac{k}{2}}}{\left( 1+x\right) ^{\frac{k}{2}\pm m}}\sum_{n=m\mp m}^{2k}%
\frac{\mathcal{G}_{n}\left( -k\mp m,-k\pm m,\sqrt{\frac{1-x}{1+x}}\right) }{%
\left( -1\right) ^{n+k}\left( 2k-n\right) !}\frac{P_{k\pm m}^{k\mp
m-n}\left( \frac{1}{x}\right) }{x^{^{-k\mp m}}}\text{,} \\
\frac{P_{k}^{\mp m}\left( \frac{1}{x}\right) }{2^{\pm m}x^{^{-k}}} &=&\frac{%
k!\left( 1-x\right) ^{\frac{k}{2}}}{\left( 1+x\right) ^{\frac{k}{2}\pm m}}%
\sum_{n=m\mp m}^{2k}\frac{\widehat{\mathcal{G}}_{n}\left( -k\mp m,-k\pm m,%
\sqrt{\frac{1-x}{1+x}}\right) }{\left( -1\right) ^{n+k}\left( 2k-n\right) !}%
P_{k\pm m}^{k\mp m-n}\left( x\right) \text{.}
\end{eqnarray*}
\end{corollary}

\section{Series relating $P_{\protect\nu }^{-\protect\mu }\left( 2x/\left(
1+x^{2}\right) \right) $ with $P_{\protect\lambda }^{-\protect\sigma }\left(
1/x\right) $ and $P_{\protect\nu }^{-\protect\mu }\left( x\right) $ with $P_{%
\protect\lambda }^{-\protect\sigma }\left( \left( 1+x^{2}\right) /2x\right) $%
}

A number of mutually inverse series relating $P_{\nu }^{-\mu }\left( \frac{2x%
}{1+x^{2}}\right) $ with $P_{\lambda }^{-\sigma }\left( \frac{1}{x}\right) $
and $P_{\nu }^{-\mu }\left( x\right) $ with $P_{\lambda }^{-\sigma }\left( 
\frac{1+x^{2}}{2x}\right) $ can be established in the same way as the
results of section 4 by noting that $2x/\left( 1+x^{2}\right) =$ $\tanh
2\alpha $ as $x=\tanh \alpha $ and making use of the following integral
representations obtained from the integral representations of section 2 by
simple changes of variables:%
\begin{eqnarray*}
I_{1}^{+}\left( 2\alpha ,\mu ,\nu \right) &=&2\int_{\alpha }^{\infty }\frac{%
e^{-2\mu t}dt}{\left( \sinh 2t-\sinh 2\alpha \right) ^{-\nu }}\text{,} \\
\text{ \ }I_{1}^{-}\left( 2\alpha ,\mu ,\nu \right) &=&2\int_{\alpha
}^{\infty }\frac{e^{-2\mu t}dt}{\left( \cosh 2t-\cosh 2\alpha \right) ^{-\nu
}}\text{,} \\
I_{2}^{+}\left( 2\alpha ,\mu ,\nu \right) &=&2\int_{\alpha }^{\infty }\frac{%
e^{-4\mu t}\cosh ^{\nu }2tdt}{\sinh ^{-\nu }\left( 2t-2\alpha \right) }\text{%
, } \\
\text{\ }I_{2}^{-}\left( 2\alpha ,\mu ,\nu \right) &=&2\int_{\alpha
}^{\infty }\frac{e^{-4\mu t}\sinh ^{-\nu }2tdt}{\sinh ^{-\nu }\left(
2t-2\alpha \right) }\text{, } \\
I_{3}^{+}\left( \alpha ,\sigma ,\nu \right) &=&2\int_{\alpha }^{\infty }%
\frac{\sinh ^{\sigma }\left( 2t-2\alpha \right) }{\cosh ^{\sigma +2\nu +2}2t}%
dt\text{,} \\
I_{3}^{-}\left( \alpha ,\sigma ,\nu \right) &=&2\int_{\alpha }^{\infty }%
\frac{\sinh ^{\sigma }\left( 2t-2\alpha \right) }{\sinh ^{\sigma +2\nu +2}2t}%
dt\text{.}
\end{eqnarray*}%
Necessary generating functions are either\ special cases of (\ref{Laur1})
expressed in terms of Gauss hypergeometric polynomials (those that are
considered above and some additional) or new generating functions that are
similar to the second family of generating functions studied in Section 3.
Also, note that the change $2x/\left( 1+x^{2}\right) =\sin \theta $, $%
0<\theta <\frac{\pi }{2}$, transforms relations between $P_{\nu }^{-\mu
}\left( x\right) $ and $P_{\lambda }^{-\sigma }\left( \frac{1+x^{2}}{2x}%
\right) $ into relations connecting $P_{\nu }^{-\mu }\left( \tan \frac{%
\theta }{2}\right) $ and $P_{\lambda }^{-\sigma }\left( \frac{1}{\sin \theta 
}\right) $ while relations between $P_{\nu }^{-\mu }\left( \frac{2x}{1+x^{2}}%
\right) $ and $P_{\lambda }^{-\sigma }\left( \frac{1}{x}\right) $ turn into
relations connecting $P_{\nu }^{-\mu }\left( \sin \theta \right) $ and $%
P_{\lambda }^{-\sigma }\left( \cot \frac{\theta }{2}\right) $.

As example, we consider mutually inverse series arising from connections
between $I_{1}^{+}\left( \alpha ,\mu ,\nu \right) $ and $I_{1}^{-}\left(
2\alpha ,\frac{\mu +\nu }{2},\nu \right) $.

\begin{theorem}
\ Let $\mu $,$\nu \in \mathbb{C}$. Then for $x\in \left( 0,1\right) $,%
\begin{equation}
P_{\nu }^{-\mu }\left( x\right) =\frac{\sqrt{\pi }}{2^{\mu -2\nu }}%
\sum_{n=0}^{\infty }\frac{\left( 2\nu \right) _{n}\left( \mu -\nu \right)
_{n}C_{n}^{1/2-\nu -n}\left( x\right) P_{\nu }^{-\frac{n+\mu +\nu }{2}%
}\left( \frac{1+x^{2}}{2x}\right) }{2^{2n}\left( \frac{1}{2}+\nu \right)
_{n}\Gamma \left( \frac{\mu -\nu +n+1}{2}\right) x^{-\nu }\left(
1-x^{2}\right) ^{\frac{n+\nu }{2}}}\text{,}  \label{con1}
\end{equation}%
\begin{equation}
\frac{\sqrt{\pi }P_{\nu }^{-\frac{\mu +\nu }{2}}\left( \frac{1+x^{2}}{2x}%
\right) }{2^{\mu -2\nu }x^{-\nu }\Gamma \left( \frac{\mu -\nu +1}{2}\right) }%
=\sum_{n=0}^{\infty }\frac{\left( -2\nu \right) _{n}\left( \mu -\nu \right)
_{n}}{2^{n}\left( \frac{1}{2}-\nu \right) _{n}}\frac{C_{n}^{1/2+\nu
-n}\left( x\right) }{\left( 1-x^{2}\right) ^{\frac{n-\nu }{2}}}P_{\nu
}^{-\mu -n}\left( x\right) \text{.}  \label{con2}
\end{equation}
\end{theorem}

\begin{proof}
Our proof is based on Theorem 4. Let $\nu \in $ $D_{\nu }\left( r\right) $ $%
=\{\left\vert \nu -1\right\vert \leq r\}\cap \{\func{Re}\nu \leq 1\}$, $1<r<%
\sqrt{2}$, $\func{Re}\mu >\func{Re}\nu $ and $\alpha >0$. In this case, by
employing the identity%
\begin{equation*}
\left( \sinh u-\sinh \alpha \right) ^{\nu }=\frac{\left( \cosh 2u-\cosh
2\alpha \right) ^{\nu }}{2^{\nu }\left( \sinh u+\sinh \alpha \right) ^{\nu }}
\end{equation*}%
and the generating function (\ref{sh-C}), we can write 
\begin{equation}
I_{1}^{+}\left( \alpha ,\mu ,\nu \right) =\int_{\alpha }^{\infty }\frac{%
2^{-\nu }e^{-u\mu }}{\left( \cosh 2u-\cosh 2\alpha \right) ^{-\nu }}\left(
\sum_{n=0}^{\infty }\mathfrak{C}_{n}\left( \alpha ,\nu \right) e^{-\left(
n+\nu \right) u}\right) du\text{.}  \label{pr1}
\end{equation}%
Taking into account (\ref{Geg-est}), we obtain as $u\geq \alpha $, 
\begin{equation*}
\left\vert \frac{e^{-u\left( \mu +\nu \right) }e^{-\left( u-\alpha \right)
n}e^{-\alpha n}\mathfrak{C}_{n}\left( \alpha ,\nu \right) }{\left( \cosh
2u-\cosh 2\alpha \right) ^{-\nu }}\right\vert \leq \frac{\mathcal{K}_{\nu
}e^{-u\func{Re}\left( \mu +\nu \right) }}{\left( \cosh 2u-\cosh 2\alpha
\right) ^{-\func{Re}\nu }}\text{,}
\end{equation*}%
where the function in the right side belongs to $L^{1}\left( \alpha ,\infty
\right) $. Then, according to the dominated convergence theorem for series,
one can integrate (\ref{pr1}) term-by-term. It leads to the relation%
\begin{equation*}
I_{1}^{+}\left( \alpha ,\mu ,\nu \right) =\sum_{n=0}^{\infty }\frac{%
\mathfrak{C}_{n}\left( \alpha ,\nu \right) }{2^{1+\nu }}I_{1}^{-}\left(
2\alpha ,\frac{\mu +\nu +n}{2},\nu \right) \text{, }\func{Re}\mu >\func{Re}%
\nu \text{, }\nu \in D_{\nu }\left( r\right) \text{,}
\end{equation*}%
which on employing (\ref{rep1}) and (\ref{rep2}), using the multiplication
formula for gamma functions and making the change $\tanh \alpha =x$ turns
into (\ref{con1}). Note that according to (\ref{as11}) and (\ref{C}), for
all $\nu ,\mu \in \mathbb{C}$, 
\begin{equation*}
\mathfrak{C}_{n}\left( \alpha ,\nu \right) P_{\nu }^{-\frac{n+\mu +\nu }{2}%
}\left( \coth 2\alpha \right) =O\left( n^{-2}\right) \text{.}
\end{equation*}%
Therefore, both sides of (\ref{con1}) are analytic functions of parameters $%
\mu ,\nu \in \mathbb{C}$, and (\ref{con1}) is valid by virtue of the
analytic continuation. The inverse series (\ref{con2}) is established in the
same manner.
\end{proof}

As $\mu -\nu \in \mathbb{N}_{0}$, series in (\ref{con1}) and (\ref{con2})
turns into the mutually inverse sums for Gegenbauer polynomials.

\begin{corollary}
For $\mu \in \mathbb{C}$,%
\begin{equation}
C_{k}^{\mu +1/2}\left( x\right) =\sum_{m=0}^{\left[ \frac{k}{2}\right] }%
\frac{\Lambda _{1}\left( k,m,\mu \right) }{x^{-m}}C_{k-2m}^{\frac{1}{2}%
-2k+2m-\mu }\left( x\right) C_{m}^{k-m+\mu +\frac{1}{2}}\left( \frac{1+x^{2}%
}{2x}\right) \text{,}  \label{Fin11}
\end{equation}%
\begin{equation}
x^{l}C_{l}^{\mu +l+1/2}\left( \frac{1+x^{2}}{2x}\right)
=\sum_{n=0}^{2l}\Lambda _{2}\left( l,n,\mu \right) C_{n}^{\frac{1}{2}+2l+\mu
-n}\left( x\right) C_{2l-n}^{\mu +n+\frac{1}{2}}\left( x\right) \text{,}
\label{Fin22}
\end{equation}%
where%
\begin{equation}
\Lambda _{1}\left( k,m,\mu \right) =\frac{\left( -1\right)
^{k+m}2^{2m}\left( 2k+2\mu \right) _{k-2m}\left( \mu +\frac{1}{2}\right)
_{k-m}}{\left( \frac{1}{2}+\mu +k\right) _{k-2m}\left( 2\mu +k+1\right)
_{k-m}}\text{,}  \notag
\end{equation}%
\begin{equation*}
\Lambda _{2}\left( l,n,\mu \right) =\frac{\left( -1\right) ^{n+l}\left(
-4l-2\mu \right) _{n}\left( \mu +l+\frac{1}{2}\right) _{n-l}}{2^{2l}\left( 
\frac{1}{2}-2l-\mu \right) _{n}\left( 2\mu +2l+1\right) _{n-l}}.
\end{equation*}
\end{corollary}

\begin{proof}
As $\mu -\nu =-k$, $k\in \mathbb{N}_{0}$, (\ref{con1}) becomes%
\begin{equation*}
P_{k+\mu }^{-\mu }\left( x\right) =\frac{\sqrt{\pi }k!}{2^{-\mu -2k}}%
\sum_{n=0}^{k}\frac{\left( -1\right) ^{n}\left( 2k+2\mu \right) _{n}C_{n}^{%
\frac{1}{2}-k-\mu -n}\left( x\right) x^{k+\mu }P_{k+\mu }^{-\frac{n+k}{2}%
-\mu }\left( \frac{1+x^{2}}{2x}\right) }{2^{2n}\left( \frac{1}{2}+\mu
+k\right) _{n}\Gamma \left( \frac{-k+n+1}{2}\right) \left( k-n\right)
!\left( 1-x^{2}\right) ^{\frac{n+k+\mu }{2}}}\text{,}
\end{equation*}%
which for even $k=2l$ turns as $n=2r$ into%
\begin{equation*}
\frac{P_{2l+\mu }^{-\mu }\left( x\right) }{2^{4l}\left( 2l\right) !}%
=\sum_{r=0}^{l}\frac{\sqrt{\pi }\left( 4l+2\mu \right) _{2r}C_{2r}^{\frac{1}{%
2}-\mu -2l-2r}\left( x\right) x^{2l+\mu }P_{l-r+l+r+\mu }^{-l-r-\mu }\left( 
\frac{1+x^{2}}{2x}\right) }{2^{2r-\mu }\left( \frac{1}{2}+\mu +2l\right)
_{2r}\left( l-r\right) !\left( 1-x^{2}\right) ^{r+l+\frac{\mu }{2}}}\text{,}
\end{equation*}%
and for odd $k=2l+1$ turns as $n=2r+1$ into%
\begin{equation*}
\frac{P_{2l+\mu +1}^{-\mu }\left( x\right) }{2^{4l}\left( 2l+1\right) !}%
=-\sum_{r=0}^{l}\frac{\sqrt{\pi }\left( 4l+2+2\mu \right)
_{2r+1}C_{2r+1}^{-2l-2r-\mu -\frac{3}{2}}\left( x\right) P_{l-r+l+r+1+\mu
}^{-l-r-1-\mu }\left( \frac{1+x^{2}}{2x}\right) }{2^{2r-\mu }\left( \frac{3}{%
2}+\mu +2l\right) _{2r+1}\left( l-r\right) !x^{-2l-1-\mu }\left(
1-x^{2}\right) ^{r+l+1+\frac{\mu }{2}}}\text{.}
\end{equation*}%
On making the change $l-r=m$ and making use of (\ref{P-C}), we obtain (\ref%
{Fin11}) by simple manipulations. The inverse sum (\ref{Fin22}) is derived
from (\ref{con2}) as $\mu -\nu =-2l$. Both relations are valid on $\mathbb{R}
$ by virtue of holomorphic continuation.
\end{proof}

Another curious relation follows from (\ref{con2}) as $\mu -\nu =-2l-1$,%
\begin{equation*}
\sum_{n=0}^{2l+1}\Lambda _{3}\left( l,n,\mu \right) C_{n}^{3/2+2l+\mu
-n}\left( x\right) C_{2l-n+1}^{\mu +n+1/2}\left( x\right) =0\text{,}
\end{equation*}%
where%
\begin{equation*}
\Lambda _{3}\left( l,n,\mu \right) =\frac{\left( -1\right) ^{n}\Gamma \left(
n-4l-2\mu \right) \Gamma \left( n+\mu +1/2\right) }{\Gamma \left( n-\frac{1}{%
2}-2l-\mu \right) \Gamma \left( n+2\mu +2l+2\right) }\text{.}
\end{equation*}

In conclusion, we note that the finite sums relations for Gegenbauer and
associated Legendre polynomials of a real variable obtained in this article
can be readily extended on certain domains in the complex plane due to
analytic continuation. For example, one can see that mutually inverse sums (%
\ref{Fin11}) and (\ref{Fin22}) are valid on the whole complex plane. For the
infinite series relations such analytic continuations, which include
extensions of Ferrers functions into the complex plane with the cuts $%
(-\infty ,-1]$ and $[1,\infty )$ by means of (\ref{Hyper1}), can be made by
employing asymptotics of the Gauss hypergeometric functions and polynomials
introduced in of section 3 to study series convergence.

\end{document}